\documentclass[10pt]{article}

\usepackage{amsmath,amsthm,amssymb}
\usepackage{caption,color,graphicx,subfig}
\usepackage[text={6.5in,9.25in},centering]{geometry}
\usepackage{paralist}
\usepackage[sort,numbers]{natbib}

\setcaptionmargin{0.25in}

\makeatletter\@addtoreset{equation}{section}\makeatother

\newtheorem{thm}{Theorem}
\newtheorem{lem}{Lemma}[section]
\newtheorem{hyp}{Hypothesis}

\theoremstyle{definition}
\newtheorem{rmk}{Remark}[section]

\def\Re{\mathop\mathrm{Re}\nolimits}

\newcommand{\revised}[1]{\textcolor{black}{#1}}

\begin{document}

\title{Localized patterns in planar bistable \revised{weakly coupled} lattice systems}

\author{
Jason J. Bramburger and Bj\"orn Sandstede\\
Division of Applied Mathematics\\
Brown University\\
Providence, RI 02912, USA
}

\date{\today}
\maketitle

\begin{abstract}
Localized planar patterns in spatially extended bistable systems are known to exist along intricate bifurcation diagrams, which are commonly referred to as snaking curves. Their analysis is challenging as techniques such as spatial dynamics that have been used to explain snaking in one space dimension no longer work in the planar case. Here, we consider bistable systems posed on square lattices and provide an analytical explanation of snaking near the anti-continuum limit using Lyapunov--Schmidt reduction. We also establish stability results for localized patterns, discuss bifurcations to asymmetric states, and provide further numerical evidence that the shape of snaking curves changes drastically as the coefficient that reflects the strength of the spatial coupling crosses a finite threshold.
\end{abstract}


\section{Introduction}

Spatially extended bistable systems have been shown to exhibit a wide variety of stationary spatial patterns. In this manuscript, we focus on localized patterns that resemble a stable patterned state in some compact, connected spatial region and a second spatially homogeneous stable rest-state outside of this compact region. Localized structures arise in many applications, for instance as urban crime spots \cite{Hotspot3,Hotspot,Hotspot2}, vegetation patterns \cite{Veg1,Veg2}, and soft matter quasicrystals \cite{Subramanian}, and in chemical reactions \cite{Chemical}, semiconductors \cite{Semiconductor}, and ferrofluids \cite{Ferrofluid}. We refer to the review papers \cite{Dawes,Knobloch} for additional references to applications.

\begin{figure}
\center
\includegraphics{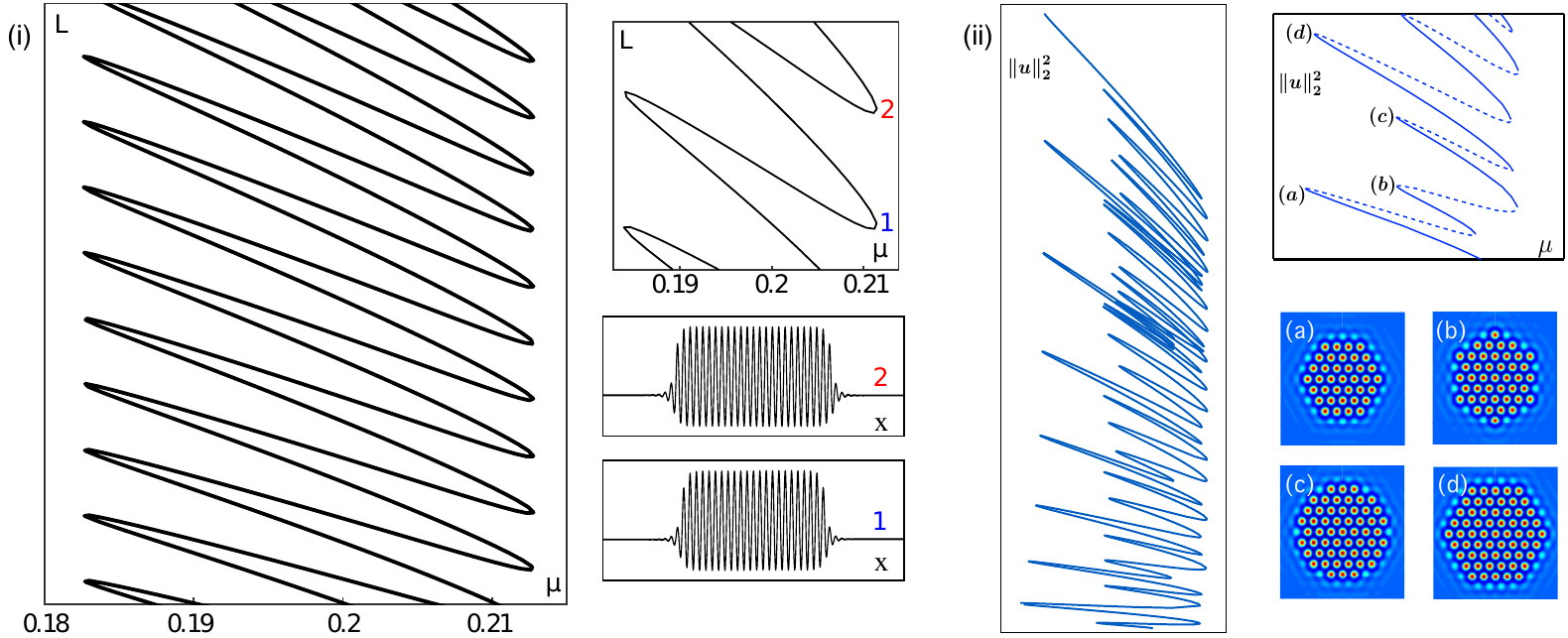}
\caption{Shown are the existence curves and sample spatial profiles of localized patterns in the Swift--Hohenberg equation in one space dimension in panel~(i) and the plane (reproduced from \cite{Lloyd}) in panel~(ii). Note that fold bifurcations are aligned with two vertical asymptotics in the 1d case, while they are aligned with many different vertical asymptotes in the planar case.}
\label{f:1}
\end{figure}

Upon varying a system parameter, localized patterns often trace out intricate existence curves, which are commonly referred to as \textit{snaking} diagrams. To illustrate these curves, we show the existence curves of one-dimensional localized roll structures and planar localized hexagon patches of the Swift--Hohenberg equation in Figure~\ref{f:1}. In both cases, the existence curves in function space are unbounded, and the spatial extent $L$ of the patterned (spatially non-homogeneous) part of the spatial profiles increases without bound along the curve. Furthermore, solutions exist only within a bounded interval in parameter space, and branches turn back at infinitely many fold bifurcations at which additional rolls or hexagon cells are added to the pattern.

The mechanisms that drive localized patterns and the associated snaking diagrams are well understood for partial differential equations (PDEs) posed on \revised{the real line or on cylindrical domains \cite{Pomeau, Woods, Coullet, Burke2, Beck, Kozyreff} and much is known also about asymmetric states \cite{Burke, Burke2}, localized states with disconnected regions of localization \cite{2Pulse}, broken symmetries \cite{Wagenknecht, Makrides, Xu}, PDE stability \cite{Makrides2}, and situations where snaking is precluded \cite{Aougab}. For lattices, results about pinning of one-dimensional and planar fronts near the continuum limit were obtained in \cite{Dean, King}, respectively, using asymptotics-beyond-all-orders methods. Results about snaking diagrams of localized patterns in one-dimensional lattices were recently established in \cite{Bramburger}.}

Despite this progress in understanding localized patterns in one spatial dimension, little is known analytically about planar patterns, whose bifurcation diagrams are more complicated and whose spatial profiles change in a more intricate way along their existence curves \cite{Chong,Lloyd2,Lloyd,McCalla,Taylor}. For instance, as shown in Figure~\ref{f:1}, localized planar hexagon patches do not grow by simply adding a complete set of hexagon cells around the entire perimeter of the current patch at each fold bifurcation. Instead, individual hexagon cells are added at each fold bifurcation, and the patterns sometimes even recede inwards from the outermost corners. The more complex changes of the profiles are reflected in the bifurcation diagrams, which are less regular and predictable compared to the one-dimensional case; see again Figure~\ref{f:1}.

To better understand the bifurcation structure of localized patterns in higher space dimensions, we will \revised{investigate} spatially extended systems posed on a planar square lattice of the form
\begin{equation}\label{LDS_Intro}
\dot{u}_{n,m} = d (\Delta u)_{n,m} + f(u_{n,m},\mu), \quad (n,m)\in\mathbb{Z}^2,
\end{equation}
where the $5$-point discrete Laplacian $\Delta$ defined by
\begin{equation}\label{e:lo}
(\Delta u)_{n,m} := u_{n+1,m} + u_{n-1,m} + u_{n,m+1} + u_{n,m-1} - 4u_{n,m}
\end{equation}
reflects the interaction across neighboring elements of the lattice $\mathbb{Z}^2$, the constant $d\geq0$ measures the strength of these interactions, and the function $f$ is a bistable nonlinearity that depends on a one-dimensional parameter $\mu$. \revised{We will concentrate on steady-state solutions of (\ref{LDS_Intro}) that are $D_4$-symmetric in the following sense. If we consider the lattice $\mathbb{Z}^2$ as consisting of those elements of the plane $\mathbb{R}^2$ that have integer coordinates, then the lattice is invariant under two representations of the group $D_4$. The first representation of $D_4$ is generated by the rotation by $90^\mathrm{o}$ around the origin and reflection across the horizontal line through the origin. The second representation is generated by rotation by $90^\mathrm{o}$ around the point $(\frac12,\frac12)$ and reflection across the horizontal line through the same point $(\frac12,\frac12)$. With these representations in mind, we will focus on on-site steady states of (\ref{LDS_Intro}), which, by definition, are invariant under the first representation of $D_4$, and off-site steady states of (\ref{LDS_Intro}), which, by definition, are invariant under the second representation of $D_4$; we refer to the center panels of Figure~\ref{f:2} for examples of off-site and on-site patterns.}

\begin{figure}
\center
\includegraphics[scale=0.82]{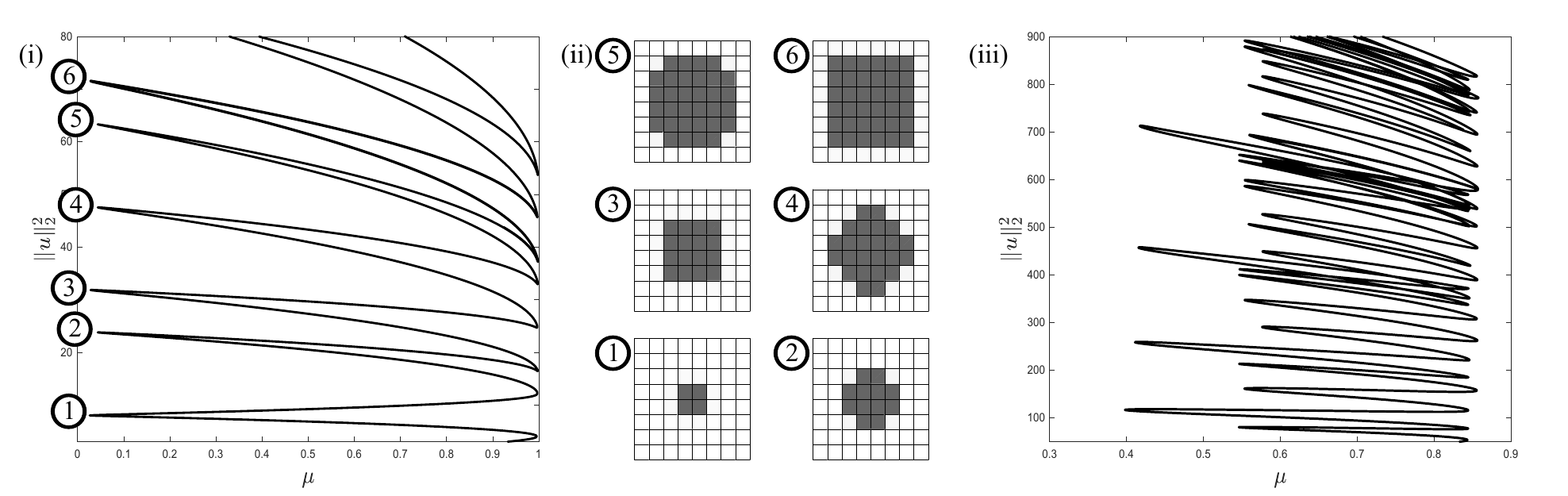}
\includegraphics[scale=0.82]{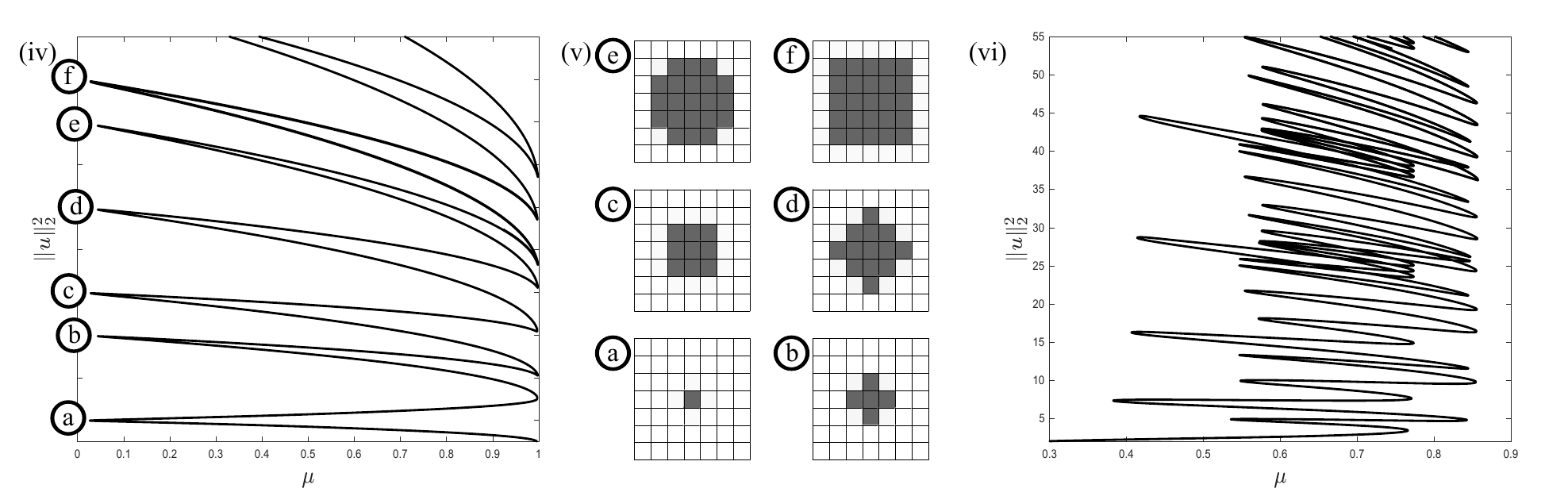}
\caption{\revised{Shown are sample profiles and snaking curves of planar localized $D_4$-symmetric off-site (top) and on-site (bottom) patterns of the lattice system (\ref{LDS_Intro}) with $f(u,\mu)=-\mu u+2u^3-u^5$ where $d=0.001$ in the left panels~(i) and~(iv), while $d=0.1$ in the right panels~(iii) and~(vi).}}
\label{f:2}
\end{figure}

Using the cubic-quintic nonlinearity $f(u,\mu)=-\mu u+2u^3-u^5$, it was demonstrated in \cite{Taylor} using numerical continuation that the bifurcation curves of $D_4$-symmetric \revised{square patches} of (\ref{LDS_Intro}) bear a striking resemblance to those of localized hexagon patches in the planar Swift--Hohenberg equation: compare Figure~\ref{f:1}(ii) showing hexagon patches of the Swift--Hohenberg equation with Figure~\ref{f:2}(iii) showing \revised{$D_4$-symmetric off-site} patterns of (\ref{LDS_Intro}). \revised{Similar numerical results were obtained recently in \cite{Susanto} for equations similar to (\ref{LDS_Intro}) posed on square, hexagonal, and triangular lattices\footnote{We were not aware of this manuscript prior to submitting this work}}. An intuitive reasoning for the similarity between the continuous and discrete case is that domain-filling hexagons are the preferred planar state in the continuous Swift--Hohenberg PDE, and we can therefore think of hexagon patches as developing on an underlying hidden hexagonal lattice. Posing the system directly on this lattice might therefore reproduce a similar bifurcation diagram.

As mentioned earlier, not much is known analytically for localized planar patterns, and one of the reasons is that the techniques used in the one-dimensional case rely primarily on formulating the existence problem as a spatial dynamical system in the unbounded \revised{spatial} variable, so that localized structures can be viewed and constructed as homoclinic orbits \cite{Avitabile,Beck,Bramburger}. This approach is no longer available for genuinely planar patterns. For lattice dynamical systems, we can exploit the fact that the anti-continuum limit of (\ref{LDS_Intro}), which corresponds to setting $d=0$, provides a regime that is accessible to analysis as the equations on individual vertices of the lattice decouple from each other. Applying Lyapunov--Schmidt reduction and blow-up techniques in the regime $0<d\ll1$, we will construct \revised{$D_4$-}symmetric localized patterns of (\ref{LDS_Intro}) and prove that their bifurcation curves resemble those of the spatially one-dimensional case as illustrated in \revised{Figure~\ref{f:2} for off-site and on-site patterns. In particular, we will see that, starting from the center of each face, new cells are added on each face as the branch passes through fold bifurcations; see Figure~\ref{f:2} for an illustration.}

The question of what causes the shift from the `regular' snaking curves for small $d$ to the `irregular' snaking curves for larger $d$ visible in Figure~\ref{f:2} was answered partially by Taylor and Dawes \cite{Taylor}, whose numerical computations revealed that increasing the coupling parameter $d$ causes a closed curve (a so-called \textit{isola}) of square patterns to collide with and attach itself to the snaking curve in a collision they termed a \textit{switchback}. We will build and expand on their investigation by providing numerical evidence that indicate that switchbacks arise near the same critical value of $d$. In summary, this paper contains the following results:
\begin{compactitem}
\item Analysis: We prove that for $0<d\ll1$ the bifurcation curves of \revised{$D_4$-}symmetric localized patterns of (\ref{LDS_Intro}) resemble the snaking curves of the spatially one-dimensional \revised{continuous and discrete} settings and determine spectral and nonlinear stability of these localized patterns.
\item Numerics: We discuss asymmetric patterns that bifurcate from the symmetric patterns and provide computations that indicate that the switchbacks, which are responsible for reorganizing the snaking curves as the coupling parameter $d$ increases, occur near a critical threshold $d_*$ of $d$. These computations also indicate that, at the switchbacks that occur further up on the snaking curve, each isola and the snaking branch undergo a more complex reorganization that involves creating new isolas in the process.
\end{compactitem}
\revised{We note that the mechanisms that drive the fold bifurcations for planar lattices near the anti-continuum limit were also found independently in \cite{Susanto}, where the leading-order interactions between neighboring sites were discussed; we will put their considerations on a rigorous footing here.}

The remainder of this paper is organized as follows. Our analytical results and numerical findings are provided in \S\ref{sec:Results}. The proofs of our analytic results in the anti-continuum limit as well as extensions of our results to a broader class of nonlinearities can be found in \S\ref{sec:Proofs}, and we conclude with a discussion in \S\ref{sec:Discussion}.
 

\section{Main results}\label{sec:Results}

\begin{figure}
\centering
\includegraphics[scale=0.95]{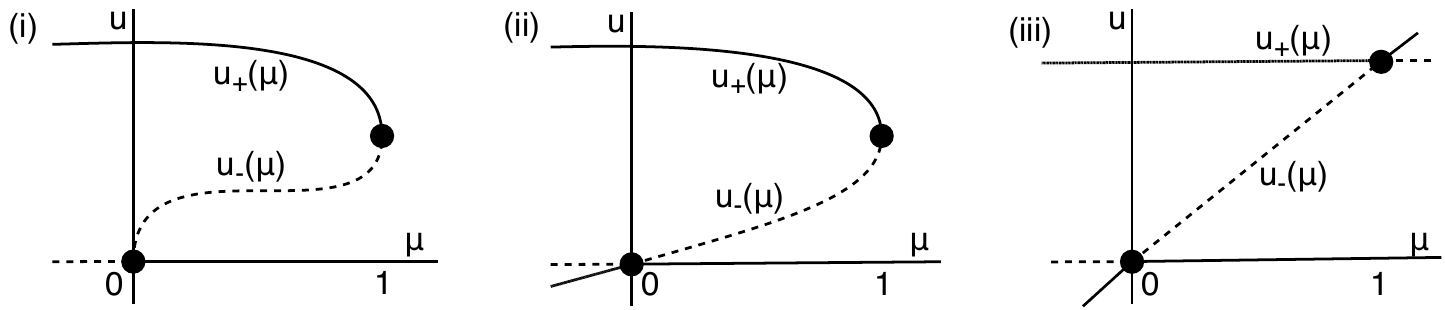}
\caption{\revised{Panels~(i)-(iii) illustrate, respectively, the zero sets of the sample functions $f(u,\mu)=-\mu u+2u^3-u^5$, $f(u,\mu)=-\mu u+2u^2-u^3$, and $f(u,\mu)=u(u-\mu)(1-u)$ in the upper half-plane to which our results apply. Solid and dashed curves indicate equilibria that are, respectively, linearly stable and unstable for $\dot{u}=f(u,\mu)$. Note that the cubic-quintic nonlinearity $f(u,\mu)=-\mu u+2u^3-u^5$ illustrated in panel~(i) satisfies Hypothesis~\ref{h1}.}}
\label{f:3}
\end{figure}

We consider the lattice dynamical system
\begin{equation}\label{LDST}
\dot{u}_{m,n} = d(\Delta u)_{n,m} + f(u_{n,m},\mu), \qquad (n,m)\in\mathbb{Z}^2
\end{equation}
posed on the Banach space $\ell^\infty(\mathbb{Z}^2)$ of bounded functions from $\mathbb{Z}^2$ into $\mathbb{R}$ equipped with the supremum norm $|\cdot|_\infty$ and the corresponding steady-state system
\begin{equation}\label{LDS}
d(\Delta u)_{n,m} + f(u_{n,m},\mu) = 0, \qquad (n,m)\in\mathbb{Z}^2,
\end{equation}
where $\Delta$ denotes the discrete Laplace operator defined in (\ref{e:lo}). \revised{Our results apply to the three classes of bistable nonlinearities $f:\mathbb{R}^2\to\mathbb{R}$ illustrated in Figure~\ref{f:3}, where the endpoints of the bistability regions correspond to pitchfork, fold, or transcritical bifurcations.}


\subsection{Analytical results}

\revised{For the sake of clarity, we focus initially on the case illustrated in Figure~\ref{f:3}(i) and make the following assumption on the nonlinearity $f:\mathbb{R}^2\to\mathbb{R}$ in the steady-state system (\ref{LDS}), which is met by the cubic-quintic nonlinearity $f(u,\mu)=-\mu u+2u^3-u^5$.}

\begin{hyp} \label{h1}
The function $f:\mathbb{R}^2\to\mathbb{R}$ is smooth and satisfies the following:
\begin{compactenum}[(i)]
\item The function $f$ is odd in $u$ so that $f(-u,\mu)=-f(u,\mu)$ for all $(u,\mu)$.
\item The set of roots of $f(u,\mu)$ is as shown in Figure~\ref{f:3}(i). In particular, for each $\mu\in(0,1)$, the function $f(u,\mu)$ has exactly three nonnegative zeros, namely $u=0$ and $u=u_\pm(\mu)$ with $0<u_-(\mu)<u_+(\mu)$, and these satisfy $f^\prime(0,\mu),f^\prime(u_+(\mu),\mu)<0<f^\prime(u_-(\mu),\mu)$.
\item At $\mu=0$, the zeros $u=0$ and $u=\pm u_-(\mu)$ collide in a generic subcritical pitchfork bifurcation.
\item At $\mu=1$, the zeros $u=u_\pm(\mu)$ collide in a generic saddle-node bifurcation.
\end{compactenum}
\end{hyp}

\begin{figure}
\centering
\includegraphics{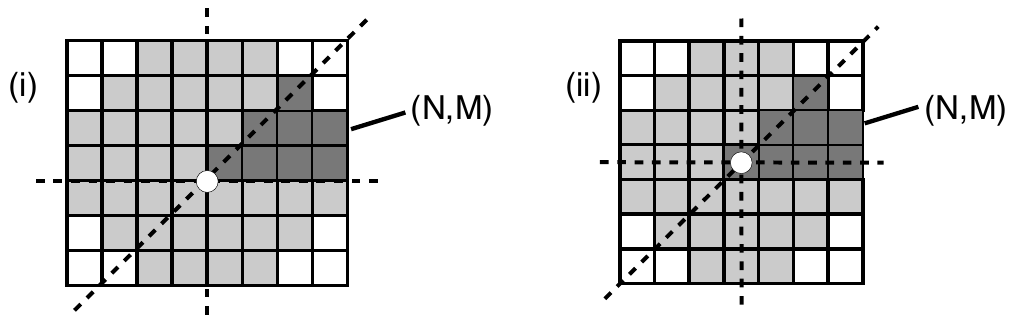}
\caption{\revised{Lattice points correspond to the midpoints of the squares shown here. Panel~(i) shows how the off-site square patch with $(N,M)=(4,2)$ is generated from the pattern $\bar{u}^{(N,M)}$ defined on the wedge $1\leq m\leq n$ by reflecting across the dashed lines through the point $(\frac12,\frac12)$ indicated by the white disk. Similarly, panel~(ii) shows how the on-site square patch with $(N,M)=(4,2)$ is generated by reflecting across the dashed lines through the point $(1,1)$ indicated by the white disk.}}
\label{f:4}
\end{figure}

\begin{figure}
\centering
\includegraphics[scale=0.4]{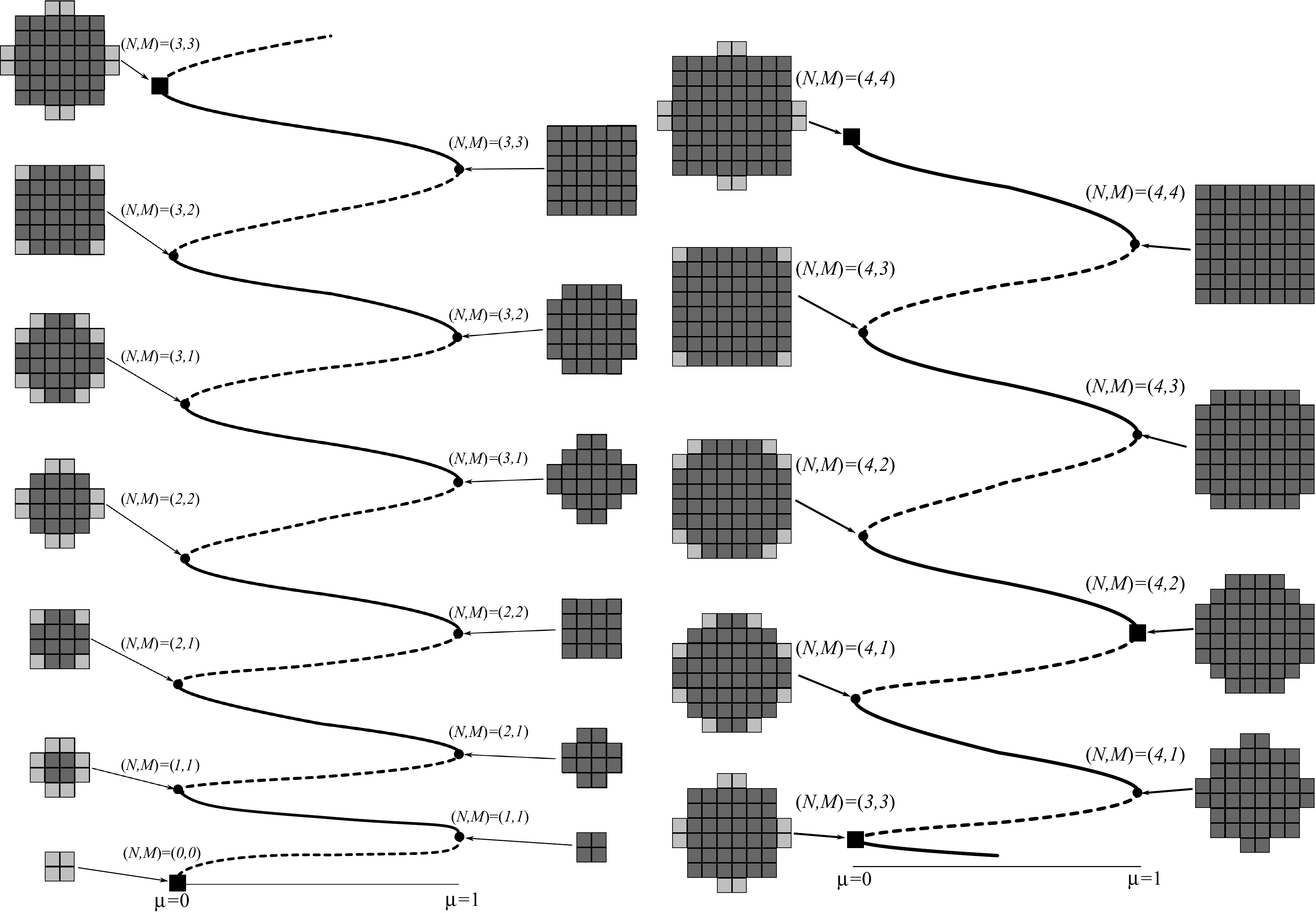}
\caption{Shown is the set $\Gamma(4)$ and the associated patterns $\bar{u}^{N,M}(\mu)$ along  the solid parts of $\Gamma(4)$ and $\bar{v}^{N,M}(\mu)$ along its dashed parts. The set $\mathcal{E}$ is indicated by black squares. Outside a small neighborhood of the fold points, the patterns along the solid  curves are stable, whilst those along the dashed curves are unstable.}
\label{f:5}
\end{figure}

We are interested in the existence of solutions to \revised{the steady-state system} (\ref{LDS}) in the anti-continuum limit $0<|d|\ll1$. We first set $d=0$ so that (\ref{LDS}) reduces to the equation $f(u_{n,m},\mu)=0$ with $(n,m)\in\mathbb{Z}^2$. Hypothesis~\ref{h1} implies that any nonnegative solution $u_{m,n}$ of $f(u_{n,m},\mu)=0$ with $0\leq\mu\leq1$ lies in the set $\{0,u_\pm(\mu)\}$. For each pair $(N,M)$ of integers with $0<M\leq N$, and each $0\leq\mu\leq1$, we define the patterns $\bar{u}^{N,M}(\mu)$ and $\bar{v}^{N,M}(\mu)$ on the wedge $0\leq m\leq n$ by
\begin{equation}\label{e:dzeropatterns}
\bar{u}^{(N,M)}_{n,m}(\mu) := \left\{\begin{array}{ll}
u_+(\mu) & 1\leq m\leq n<N \\
u_+(\mu) & n=N,\; 1\leq m\leq M, \\
0        & \mathrm{otherwise}
\end{array}\right.
\qquad
\bar{v}^{(N,M)}_{n,m}(\mu) := \left\{ \begin{array}{ll}
u_+(\mu) & 1\leq m\leq n<N \\
u_+(\mu) & n=N,\; 1\leq m<M \\
u_-(\mu) & n=N,\; m=M \\
0        & \mathrm{otherwise.}
\end{array}\right.
\end{equation}
Note that $N$ corresponds roughly to the horizontal extent of the pattern and $M$ to the vertical height in the $N$th column. \revised{The patterns defined in (\ref{e:dzeropatterns}) are pairwise distinct for $0<\mu<1$. Hypothesis~\ref{h1}(iii)-(iv), see also Figure~\ref{f:3}(i), implies that $\lim_{\mu\searrow0}u_-(\mu)=u_+(0)$ and $\lim_{\mu\nearrow1}u_-(\mu)=u_+(1)$, and we conclude that these patterns connect} to each other as follows at $\mu=0,1$:
\[
\begin{array}{lcl}
\bar{u}^{(N,M)}(0) = \bar{v}^{(N,M+1)}(0) & \quad & 1\leq M\leq N-1 \\
\bar{u}^{(N,N)}(0) = \bar{v}^{(N+1,1)}(0) && 1\leq N \\
\bar{u}^{(N,M)}(1) = \bar{v}^{(N,M)}(1)   && 1\leq M\leq N.
\end{array}
\]
\revised{As indicated in Figure~\ref{f:4}, we can extend the patterns defined in (\ref{e:dzeropatterns}) to $D_4$-symmetric off-site or on-site patterns defined on $\mathbb{Z}^2$ by reflecting profiles across the diagonal, horizontal, and vertical lines through the points $(\frac12,\frac12)$ and $(1,1)$, respectively.} Finally, for each integer $N_*>1$, we define the set
\[
\Gamma(N_*) := \bigcup_{1\leq M\leq N\leq N_*} \bigcup_{0\leq\mu\leq1} \left\{ (\bar{u}^{(N,M)}(\mu),\mu),\; (\bar{v}^{(N,M)}(\mu),\mu) \right\} \subset \ell^\infty(\mathbb{Z}^2)\times\mathbb{R}
\]
and note \revised{that this set forms a smooth curve with end points given by $u=0$ and $u=\bar{u}^{(N_*,N_*)}(0)$}. We also define the discrete set
\[
\mathcal{E} =
\left\{ (\bar{u}^{(N,N)}(0),0):\; 3\leq N \right\} \cup
\left\{ (\bar{u}^{(N,M)}(1),1):\; 2\leq M\leq N-2 \right\}
\subset \ell^\infty(\mathbb{Z}^2)\times\mathbb{R}
\]
of patterns near which we are not able to prove persistence \revised{and refer to Remarks~\ref{r:1} and~\ref{r:2} for a discussion of why our proofs fail near these bifurcation points}. We can now formulate our main theorem and refer to Figure~\ref{f:5} for an illustration of this result.

\begin{thm}\label{t1}
Assume that $f$ satisfies Hypothesis~\ref{h1}. For each $\delta_*>0$ and each integer $N_*\geq2$, there is a constant $d_*>0$ such that the following is true for each $0<d<d_*$:
\begin{compactitem}
\item Persistence: Each component of the set $U_{\delta_*}(\Gamma(N_*))\setminus U_{2\delta_*}(\mathcal{E})$ contains two unique, nonempty, continuous branches of, respectively, on-site and off-site $D_4$-symmetric solutions of \revised{the steady-state system} (\ref{LDS}). Furthermore, these branches are smooth and $C^1$-close to $\Gamma(N_*)$ for each fixed $d$, depend smoothly on $d$, and their limit as $d\searrow0$ is contained in $\Gamma(N_*)$.
\item Nonlinear stability: For $\mu\in(\delta_*,1-\delta_*)$, the patterns emerging from $\bar{u}^{(N,M)}(\mu)$ are nonlinearly stable for \revised{the dynamical system} (\ref{LDST}) posed on $\ell^\infty(\mathbb{Z}^2)$, whilst those emerging from $\bar{v}^{(N,M)}(\mu)$ are linearly unstable with eight unstable eigenvalues for $M\neq N-1$ and four unstable eigenvalues for $M=N-1$.
\item Transverse crossing: The linearization of \revised{the dynamical system}  (\ref{LDST}) about the $D_4$-symmetric patterns posed on $\ell^\infty(\mathbb{Z}^2)$ has precisely four ($M=N-1$) or eight ($M\neq N-1$) eigenvalues that cross transversely through the origin near each fold bifurcation, and there are no other eigenvalues on the imaginary axis.
\end{compactitem}
\end{thm}

In \S\ref{s2.2}, we will discuss the existence of asymmetric branches that emerge from the four or eight eigenvalues that cross the origin near each fold bifurcation.

Note that Hypothesis~\ref{h1} requires a pitchfork bifurcation at $\mu=0$ and a fold bifurcation at $\mu=1$. The conclusions of Theorem~\ref{t1} hold more generally when the system exhibits a generic transcritical bifurcation at $\mu=0$ or at $\mu=1$, and we provide the proof of Theorem~\ref{t1} for these cases in \S\ref{sec:app}.


\subsection{Numerical results}\label{s2.2}

\revised{We illustrate our analytical results with numerical computations of the system}
\begin{equation}\label{LDS_SH}
d\Delta u_{n,m} - \mu u_{n,m} + 2u_{n,m}^3 - u_{n,m}^5 = 0, \qquad
(n,m)\in\mathbb{Z}^2.
\end{equation}
\revised{Note that the cubic-quintic nonlinearity $f(u,\mu)=-\mu u+2u^3-u^5$ appearing in (\ref{LDS_SH}) satisfies Hypothesis~\ref{h1}: in particular, it is odd in $u$ and has roots given by $u=0$ and $u_\pm(\mu)=\sqrt{1\pm\sqrt{1-\mu}}$.}

\revised{We comment briefly on the numerical algorithms we implemented to solve (\ref{LDS_SH}) numerically. To compute $D_4$-symmetric profiles, we use the computational domain $\{(n,m): 1\leq m\leq n\leq N_\mathrm{d}\}$ shown in dark in Figure~\ref{f:4} with $N_\mathrm{d}=50$ with Neumann boundary conditions at the right boundary $n=N_\mathrm{d}$ and the $D_4$-symmetry conditions explained in Figure~\ref{f:4} at the bottom boundary $m=1$ and the diagonal $n=m$. We started with initial profiles at $d=0$ and then continued in $d$ or $\mu$ using a secant continuation code. Stability of profiles along each branch was assessed by computing the spectrum of the linearization of (\ref{LDS_SH}) evaluated at the full profile extended to a square of size $2N_\mathrm{d}\times2N_\mathrm{d}$ with Neumann boundary conditions using the routine \textsc{eig} in \textsc{matlab}. Similarly, asymmetric solutions were computed on a square of size $2N_\mathrm{d}\times2N_\mathrm{d}$ with Neumann boundary conditions at the boundary. The initial asymmetric profiles were constructed based on the expected symmetries derived and discussed below.}

\begin{figure}
\centering
\includegraphics[scale=0.9]{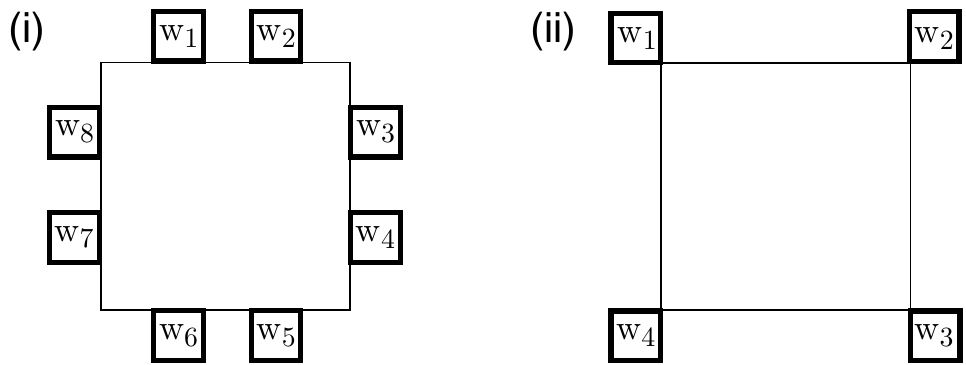}
\caption{Indicated are the eigenfunctions belonging to the critical eight (left) and four (right) eigenvalues of $D_4$-symmetric patterns that cross the origin near each fold bifurcations for $M\neq N-1$ and $M=N-1$, respectively. The critical eigenspace is parametrized by $(w_1,\ldots,w_8)\in\mathbb{R}^8$ and $(w_1,\ldots,w_4)\in\mathbb{R}^4$, respectively, and the $D_4$-symmetry acts by rotations by $90^\circ$ and reflections across the diagonal, horizontal, and diagonal lines.}
\label{f:6}
\end{figure}

\begin{figure}[th!]
\centering
\includegraphics[width=0.95\textwidth]{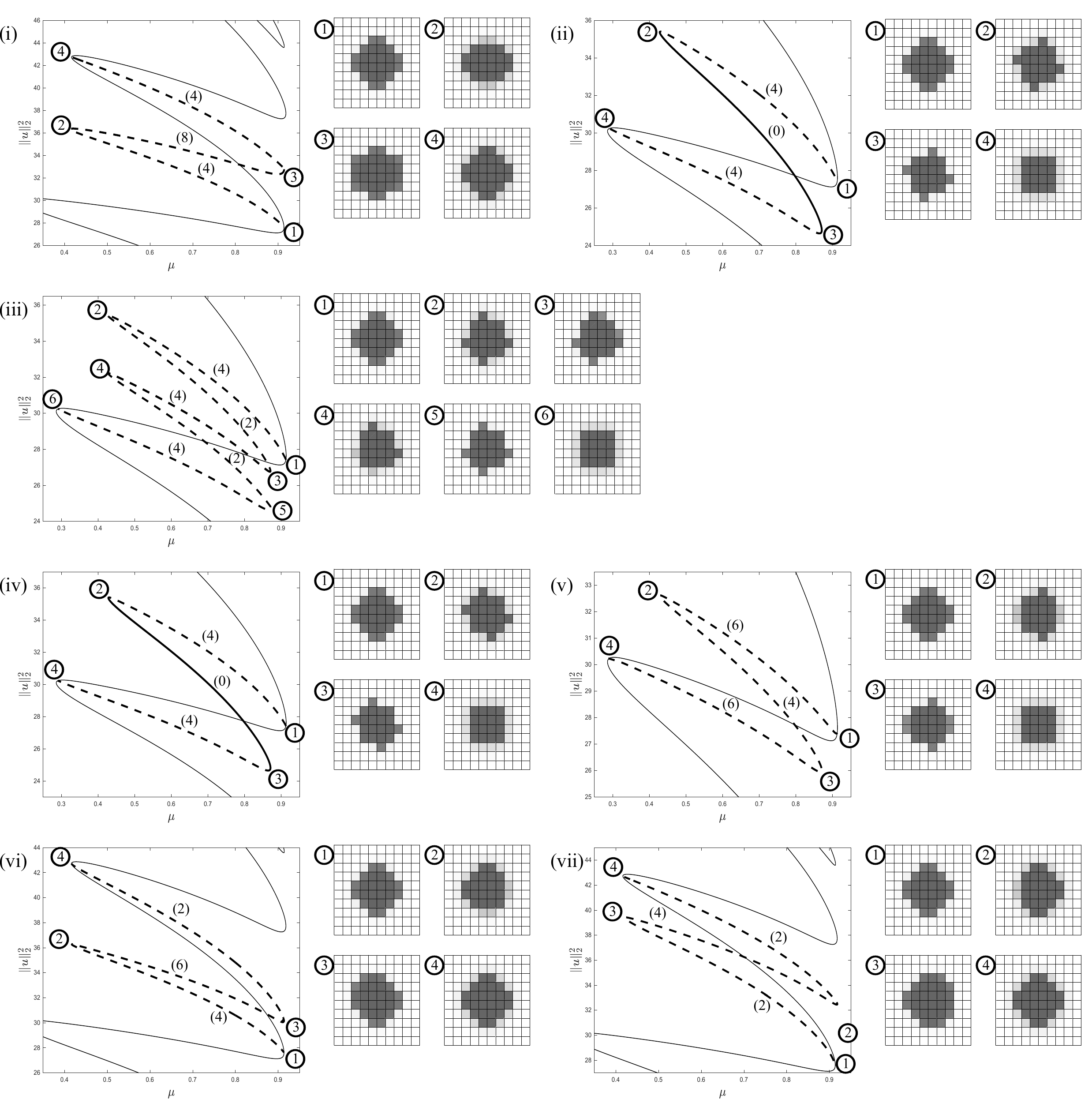}
\caption{Panel~(i)-(vii) show seven branches of asymmetric square patterns that bifurcate from the primary \revised{snaking branch (shown in light solid) taken from Figure~\ref{f:2}} near the off-site pattern $\bar{u}^{(3,1)}$ for $\mu\approx1$. The branches in panels (i)-(iii) bifurcate in three distinct one-dimensional irreducible representations of $D_4$ and are guaranteed by the equivariant branching lemma. The branches in panels (iv)-(v) and (vi)-(vii) bifurcate in two orthogonal planes on which $D_4$ acts via its two-dimensional irreducible representation. \revised{Numerical computation of spectra of the linearization about these patterns allowed us to identify stable (solid) and unstable (dashed) branches; the numbers in parentheses give the number of unstable eigenvalues of the linearization.}}
\label{f:7}
\end{figure}

\paragraph{Asymmetric branches.}
\revised{First, we focus on the anti-continuum limit $0<d\ll1$ of (\ref{LDS_SH}) and investigate bifurcations to asymmetric patterns near each of the fold bifurcations described in Theorem~\ref{t1}.} Near each fold bifurcation, precisely four ($M=N-1$) or eight ($M\neq N-1$) eigenvalues of the linearization about the $D_4$-symmetric patterns described in Theorem~\ref{t1} cross the origin transversely. Equivariant bifurcation theory \revised{\cite{Hoyle, GSS2}} implies that the symmetry group $D_4$ leaves the eigenspaces associated with these eigenvalues invariant, and we illustrate the corresponding eigenfunctions and the action of $D_4$ on these spaces in Figure~\ref{f:6}. 

We focus first on the case $M\neq N-1$. Using the form of the $D_4$-action, it is not difficult to see that the eight-dimensional eigenspace is the orthogonal sum of four one-dimensional subspaces on which $D_4$ acts with its four unique distinct one-dimensional irreducible representations and a four-dimensional subspace on which $D_4$ acts with its unique two-dimensional irreducible representation. The equivariant branching lemma \revised{(see, for instance, \cite[Theorem~3.3 in \S{XIII.3}]{GSS2} or \cite[Theorem~4.4 in \S4.2]{Hoyle})} guarantees the existence of four distinct branches in each of the one-dimensional subspaces: one of these branches is the primary branch that undergoes a fold, and we show numerical computations of the remaining three branches that emerge near the right fold for $(N,M)=(3,1)$ in Figure~\ref{f:7}. Within the remaining four-dimensional eigenspace with the two-dimensional irreducible representation of $D_4$, we expect that there should generically be two distinct pairs of eigenvalues that cross independently (see, for instance, \cite{Hoyle, GSS2}): each of these two bifurcations will lead generically to two asymmetric branches, giving a total of four branches, which are shown again in Figure~\ref{f:7} for $(N,M)=(3,1)$. Note that each bifurcating branch returns to the primary branch near a different fold. We also point out that, though each branch is initially unstable, some of the branches stabilize later (see panels~(ii) and~(iv) in Figure~\ref{f:7}).

Next, consider the case $M=N-1$: it is again not difficult to see that the isotypical decomposition of the four-dimensional critical eigenspace consists of two distinct one-dimensional representations and the two-dimensional representation of $D_4$. We therefore expect four distinct branches to emerge. We did not compute these branches.

\begin{figure}
\centering
\includegraphics[scale=0.8]{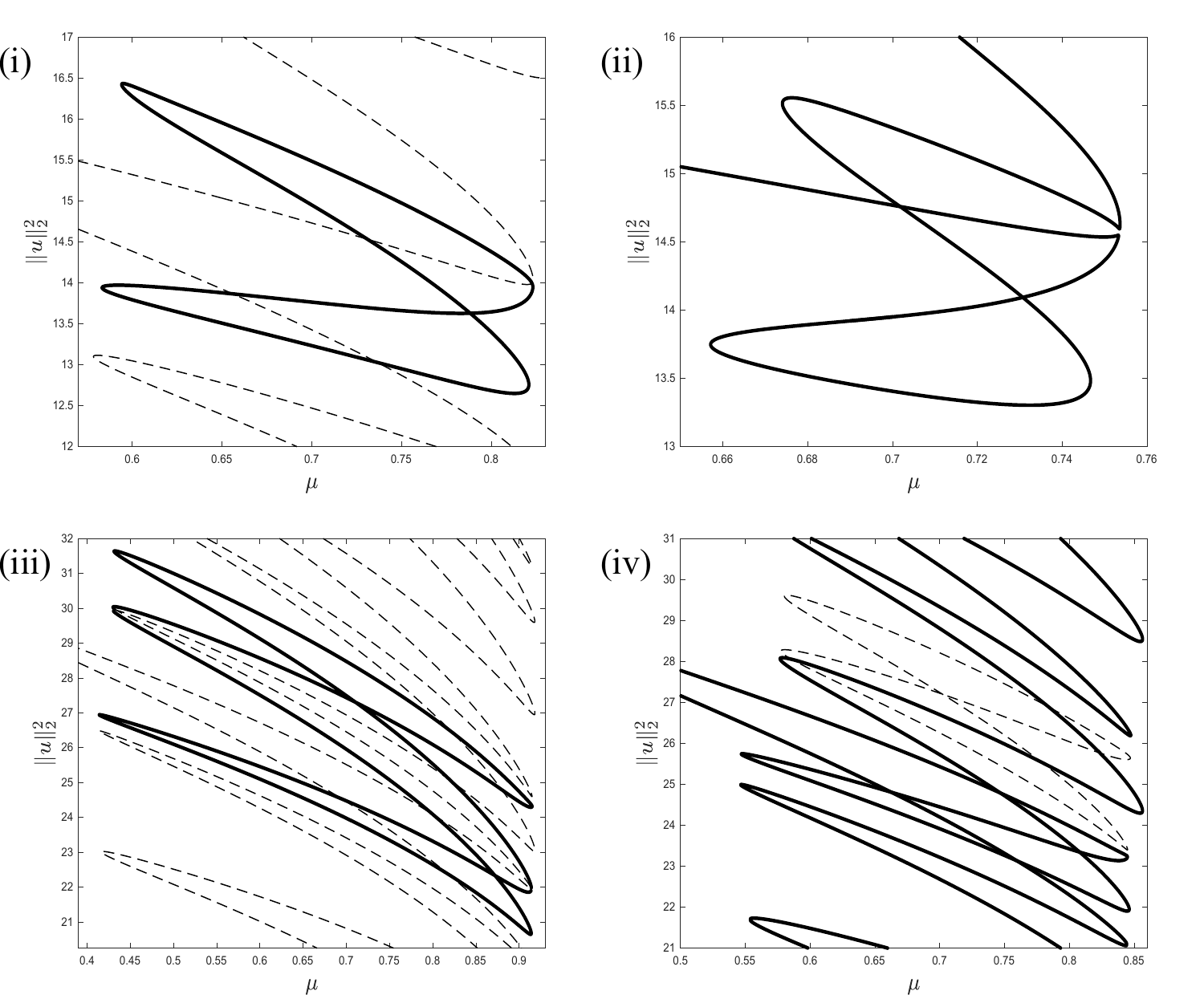}
\caption{Panel~(i) shows an isola (solid) for $d=0.12$ that is about to collide with the primary solution branch (dashed) at the rightmost fold of the pattern $\bar{u}^{(4,1)}$. Panel~(ii) shows the rearranged branch (solid) after collision for $d=0.2$. Similarly, panel~(iii) shows an isola (solid) for $d=0.05$ that is about to collide with the primary branch (dashed) at the rightmost fold of the pattern $u^{(5,1)}$. Panel~(iv) shows the rearranged branch (solid) and a new isola (dashed) that emerges after the collision for $d=0.1$.}
\label{f:8}
\end{figure}

\paragraph{Switchbacks.}
We revisit Figure~\ref{f:2} and observe that the bifurcation diagram described in Theorem~\ref{t1} agrees very well with the numerical diagram shown in Figure~\ref{f:2}(i) for $d=0.001$. In particular, even though our analysis is not able to verify the continuation through all folds along the branch, the numerical computations indicate that for sufficiently small $d > 0$ patterns grow in a very regular fashion by first adding new cells at the middle of each face and then adding additional cells to either side at each pair of folds.

\begin{figure}
\centering
\includegraphics[scale=0.8]{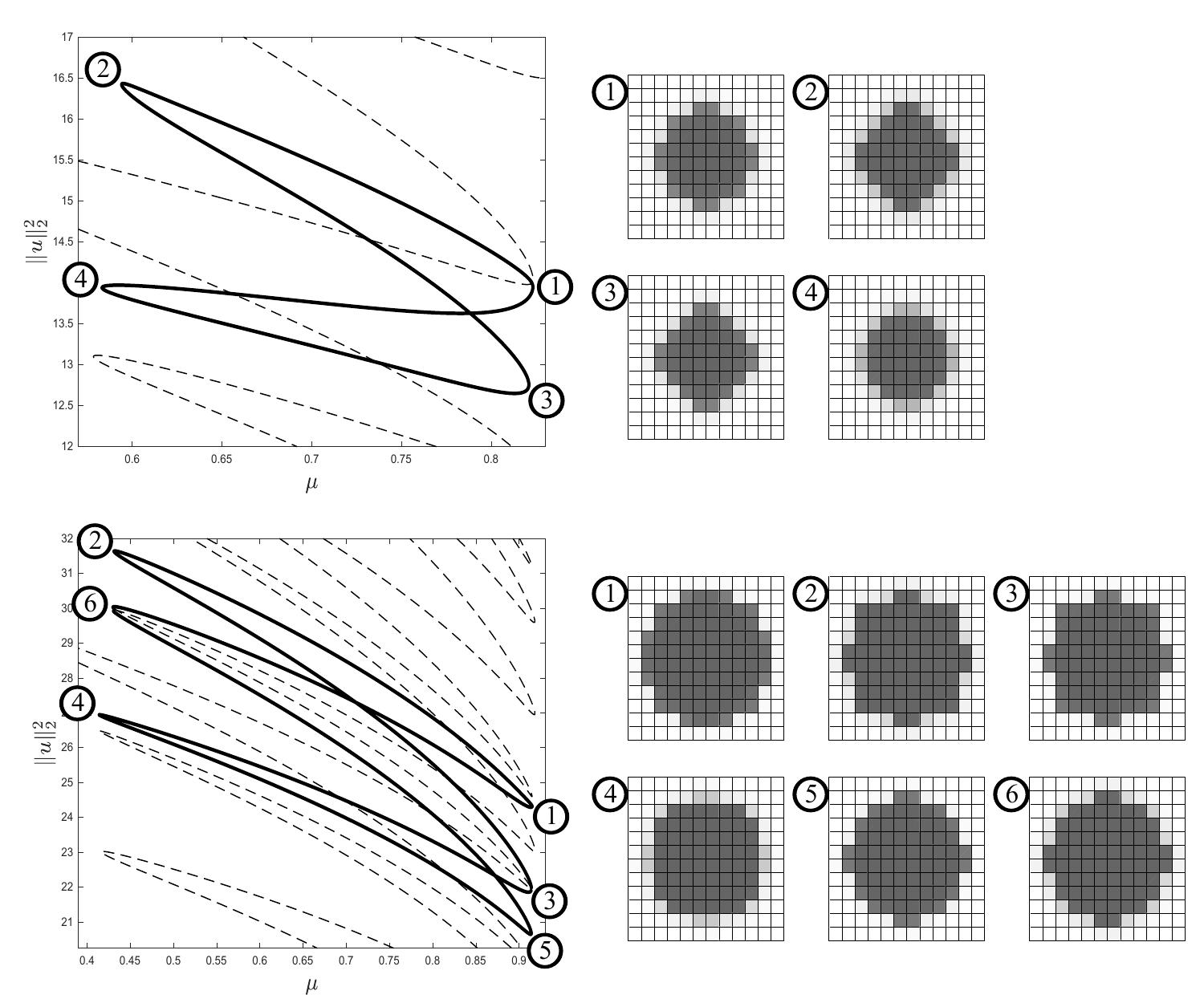}
\caption{The top and bottom panels show the solution profiles along the isolas from Figure~\ref{f:8}(i) and~(iii), respectively.}
\label{f:9}
\end{figure}

Next, we focus on when and how the regular bifurcation diagram that appears for $0<d\ll1$ changes to the more complicated diagram shown in Figure~\ref{f:2}(iii) for larger $d$. Taylor and Dawes \cite{Taylor} observed that the bifurcation branch begins to turn back on itself when isolas collide with the primary solution branch upon increasing $d$. More precisely, the collision is a codimension-two cusp bifurcation where a fold along the isola and a fold along the primary solution branch collide with each other. Taylor and Dawes also showed that several such switchbacks occur for nearby values of the coupling strength $d$. The collision process is illustrated in Figure~\ref{f:8}(i)-(ii) where we see that an isola and the primary solution branch collide at the rightmost fold of the pattern $\bar{u}^{(4,1)}$: the collision is followed by a rearrangement of the two branches into a single more complex branch structure. Figure~\ref{f:8}(iii)-(iv) shows that a similar collision of an isola and the primary branch arises near the rightmost fold of the pattern $\bar{u}^{(5,1)}$: in this case, the collision takes place simultaneously at two folds, leading to a rearrangement of the solution branch and the appearance of a new isola. In Figure~\ref{f:9}, we show the profiles of the patterns along the isolas that collide with $\bar{u}^{(N,1)}$ for $N=4,5$ just prior to their collisions.

\begin{figure}
\centering
\includegraphics{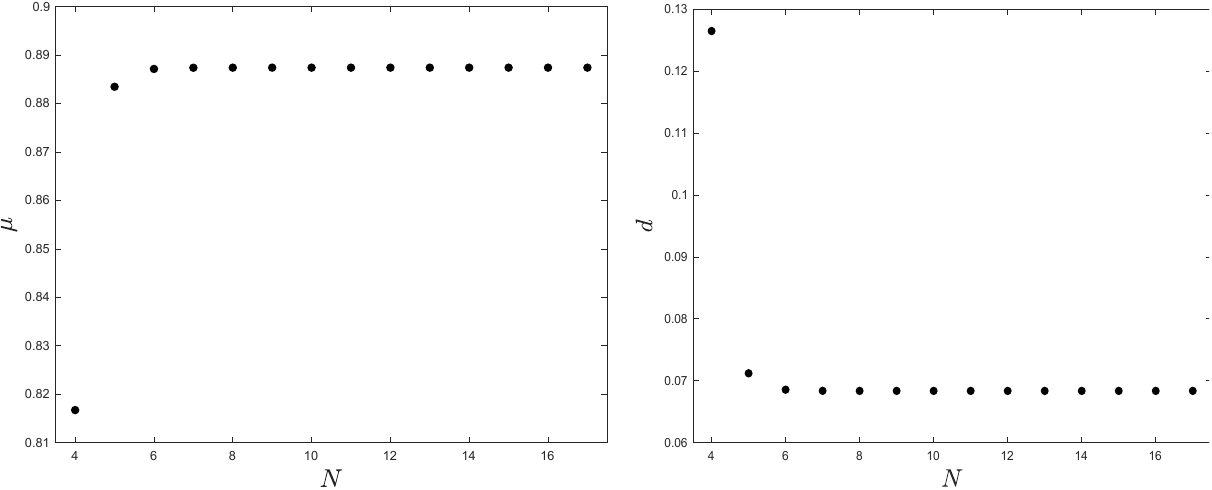}
\caption{Shown are the locations of the codimension-two cusp bifurcations caused by collisions of isolas with the solution branch at the rightmost fold bifurcations of the patterns $\bar{u}^{N,1}$ in $(N,\mu)$-space (left) and $(N,d)$-space (right). Note that the sequence $(\mu_N,d_N)$ appears to converge exponentially to the point $(\mu_\infty,d_\infty)=(0,887,0.068)$ as $N$ increases.}
\label{f:10}
\end{figure}

Our goal here is to compute the locations of the cusp bifurcations more systematically. To do so, we set up an extended system for the variables $(u,\mu,d)$ that consists of finding roots $u$ of (\ref{LDS_SH}) for parameter values $(\mu,d)$ at which the linearization of (\ref{LDS_SH}) in $u$ evaluated at $u$ has a two-dimensional null space: each solution to this extended problem corresponds to a cusp bifurcation point. We use Newton's method with initial conditions given by the profiles $\bar{u}^{(N,1)}$ and parameter values $\mu$ at fold points along the primary branch. Using this algorithm, we identified a series of cusp bifurcations near the rightmost folds of the profiles $\bar{u}^{(N,1)}$ for $N=4,\ldots,16$. Figure~\ref{f:10} illustrates our findings and shows, in particular, that the corresponding parameter values $(\mu_N,d_N)$ appear to converge exponentially to a point $(\mu_\infty,d_\infty)$ as $N$ increases. This indicates that the transition from the regular bifurcation structure observed in Figure~\ref{f:2}(i) to the more complex diagram shown in Figure~\ref{f:2}(iii) appears quite suddenly as a critical threshold $d_\infty$ of the coupling strength $d$ is crossed.


\section{Proof of Theorem~\ref{t1}}\label{sec:Proofs}

In this section, we prove Theorem~\ref{t1}. Recall that we consider the lattice dynamical system (\ref{LDST})
\[
\dot{u}_{m,n} = d(\Delta u)_{n,m} + f(u_{n,m},\mu), \qquad (n,m)\in\mathbb{Z}^2.
\]
Defining the function
\[
\mathcal{F}: \;
\ell^\infty(\mathbb{Z}^2)\times\mathbb{R}\times\mathbb{R}\longrightarrow\ell^\infty(\mathbb{Z}^2), \quad
(u,\mu,d) \longmapsto \mathcal{F}(u,\mu,d), \quad
\mathcal{F}(u,\mu,d)_{n,m} := d(\Delta u)_{n,m} + f(u_{n,m},\mu),
\]
we see that the steady-state system (\ref{LDS})
\[
d(\Delta u)_{n,m} + f(u_{n,m},\mu) = 0, \qquad (n,m)\in\mathbb{Z}^2
\]
corresponding to \revised{the dynamical system} (\ref{LDST}) is given by $\mathcal{F}(u,\mu,d)=0$. Note that $\mathcal{F}$ is smooth in its arguments.

In \S\ref{s3.1}, we will show that each pattern that satisfies \revised{the steady-state system} (\ref{LDS}) with $d=0$ for some $0<\mu<1$ can be continued uniquely into the region $0<|d|\ll1$ and determine its stability properties with respect to \revised{the lattice dynamical system} (\ref{LDST}): this result shows that it suffices to understand how branches continue through $\mu=0$ and $\mu=1$. In \S\ref{s3.2}, we will characterize $D_4$-symmetric patterns to simplify the subsequent analysis. We then consider continuation through $\mu=0$ in \S\ref{s3.3}, continuation through $\mu=1$ in \S\ref{s3.4}, and extensions to transcritical bifurcations at $\mu=0,1$ in \S\ref{sec:app}.


\subsection{Continuation for $0<\mu<1$}\label{s3.1}

Our first result shows that any solution $u$ of \revised{the steady-state system} (\ref{LDS}) with $d=0$ and $0<\mu<1$ can be continued uniquely and smoothly into the region $0<|d|\ll1$.

\begin{lem}\label{lem:IFT}
Assume that Hypothesis~\ref{h1} is met. Let $K\subset(0,1)$ be a compact interval and choose a continuous function  $v^*:K\to\ell^\infty(\mathbb{Z}^2)$ so that $v^*_{n,m}(\mu)\in\{0,u_\pm(\mu)\}$ for all $\mu\in K$, then there are positive constants $d_*,\delta>0$ and a smooth function $u^*:K\times(-d_*,d_*)\to\ell^\infty(\mathbb{Z}^2)$ with $u^*(\mu,0)=v^*(\mu)$ for each $\mu\in K$ so that the following is true:
\begin{compactitem}
\item Persistence: $\mathcal{F}(u^*(\mu,d),\mu,d)=0$ for all $(\mu,d)\in K\times(-d_*,d_*)$.
\item Uniqueness: If $(u,\mu,d)\in\ell^\infty(\mathbb{Z}^2)\times K\times(-d_*,d_*)$ satisfies $\mathcal{F}(u,\mu,d)=0$ and $|u-v^*(\mu)|<\delta$, then $u=u^*(\mu,d)$.
\item Stability: If $N_-:=\#\{(n,m)\in\mathbb{Z}^2:v^*_{n,m}(\mu)=u_-(\mu)\}$ is finite, then the linearization $\mathcal{F}_u(u^*(\mu,d),\mu,d)$ posed on $\ell^\infty(\mathbb{Z}^2)$ has precisely $N_-$ eigenvalues in the open right half-plane and none on the imaginary axis.
 \end{compactitem}
\end{lem}

\begin{proof}
Set $d=0$, then $\mathcal{F}(v^*(\mu),\mu,0)=0$ for all $\mu\in K$. Furthermore, $\mathcal{F}_u(v^*(\mu),\mu,0)$ is given by
\[
(\mathcal{F}_u(v^*(\mu),\mu,0)v)_{n,m}=f_u(v^*_{n,m}(\mu),\mu)v_{n,m},
\]
where $f_u(v^*_{n,m}(\mu),\mu)\neq0$ due to $v^*_{n,m}(\mu)\in\{0,u_\pm(\mu)\}$ and Hypothesis~\ref{h1}. The persistence and uniqueness statements now follow from the implicit function theorem. The statement about the spectrum follows since $f_u(u_-(\mu),\mu)>0$ whilst $f_u(u,\mu)<0$ for $u=0,u_+(\mu)$.
\end{proof}

In particular, we can apply this lemma to prove that the set $\Gamma(N_*)\cap\left(\ell^\infty(\mathbb{Z}^2)\times(\delta_*,1-\delta_*)\right)$ can be continued uniquely and smoothly into $0<|d|<d_*$ for an appropriate $d_*$ that depends on $\delta_*$ and $N_*$. The lemma also proves the assertion about nonlinear stability made in Theorem~\ref{t1}. It therefore suffices to show how the resulting branches can be continued through $\mu=0$ and $\mu=1$ and to discuss the transverse crossing of eigenvalues near these parameter values.


\subsection{$D_4$-symmetric patterns}\label{s3.2}

We are interested in $D_4$-symmetric off-site and on-site patterns. Off-site patterns are generated by the reflections across the horizontal, vertical, and diagonal lines on the lattice $\mathbb{Z}^2$ that pass through the point $(1/2,1/2)$ or, in other words, by the reflections $(n,m)\mapsto(1-n,m)$ and $(n,m)\mapsto(m,n)$ on the lattice $\mathbb{Z}^2$. Similarly, on-site patterns are generated by the reflections across the horizontal, vertical, and diagonal lines on $\mathbb{Z}^2$ that pass through $(1,1)$, that is by the reflections $(n,m)\mapsto(2-n,m)$ and $(n,m)\mapsto(m,n)$ on $\mathbb{Z}^2$. In particular, there is a 1:1 correspondence between $D_4$-symmetric patterns and functions $u_{n,m}$ defined for indices $(n,m)$ in the index set
\begin{equation}\label{e:I}
I = \{(n,m)\in\mathbb{Z}^2:\ 1\leq m\leq n\}
\end{equation}
by extending the latter to indices in $\mathbb{Z}^2$ using the appropriate $D_4$-action defined above. In the remainder of the proof of Theorem~\ref{t1}, we will make extensive use of this correspondence.


\subsection{Continuation through bifurcations near $\mu=0$}\label{s3.3}

We consider the behavior of $D_4$-symmetric patterns near $\mu=0$. Lemma~\ref{lem:Pitchfork} provides results for $\bar{u}^{(N,M)}$ with $1\leq M\leq N-1$, while Lemma~\ref{lem:Pitchfork2} covers the case of $\bar{u}^{(N,N)}$ with $N=1,2$. We recall that we cannot prove continuation through $\mu=0$ for $\bar{u}^{(N,N)}$ when $N\geq3$. Using a change of coordinates, we can bring the Taylor expansion of $f(u,\mu)$ about $(0,0)$ into the form
\begin{equation}\label{NormalForm1}
f(u,\mu) = -\mu u + u^3 + \mathcal{O}(\mu^2 u + \mu u^3 + u^5)
\end{equation}
and may also assume that $u_+(0)=1$.

\begin{figure}
\centering
\includegraphics{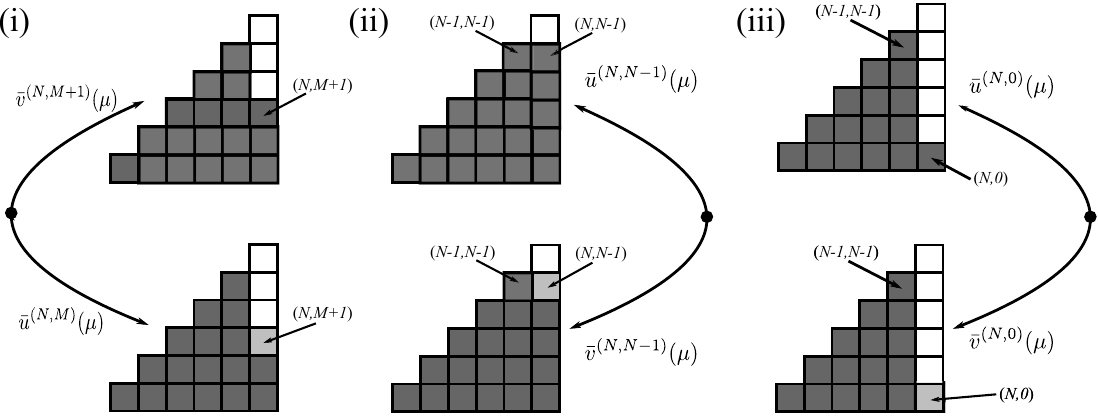}
\caption{Panel~(i) contains an illustration of the proof of Lemma~\ref{lem:Pitchfork}. Panel~(ii) shows a cartoon of the fold bifurcation near $\mu=1$ that connects $\bar{v}^{(N,M)}(\mu)$ and $\bar{u}^{(N,M+1)}(\mu)$ for the case $M=N-1$: the discussion surrounding (\ref{Saddle3}) shows that the element at $(n,m)=(N-1,N-1)$ is the obstacle to extending the results in Lemma~\ref{lem:Saddle} to $1\leq M\leq N-2$. Panel~(iii) illustrates the proof of Lemma~\ref{lem:Saddle2}.}
\label{f:11}
\end{figure}

\begin{lem}\label{lem:Pitchfork}
Fix $N \geq 2$ and $1 \leq M \leq N-1$, then the following is true separately for on-site and off-site $D_4$-symmetric solutions of \revised{the steady-state system} (\ref{LDS}). There are constants $d_1,\mu_1 > 0$ and a smooth function $\mu_l:[0,d_l] \to [0,\mu_1]$ such that the following is true for each $d \in (0,d_1]$:
\begin{compactitem}
\item Fold bifurcations: There is a pair of symmetric solutions $u_l(\mu,d)$ and $v_l(\mu,d)$ of \revised{the steady-state system} (\ref{LDS}) that bifurcate at a fold bifurcation at $\mu = \mu_l(d)$ and exist for all $\mu \in [\mu_l(d),\mu_1]$. These solutions are smooth in $(\mu,d)$, and for each fixed $\mu$ we have $u_l(\mu,d) \to \bar{u}^{(N,M)}(\mu)$ and $v_l(\mu,d) \to \bar{v}^{(N,M+1)}(\mu)$ as $d\searrow0$.
\item Expansion: The function $\mu_l(d)$ satisfies $\mu_l(d) = 3d^\frac{2}{3} + \mathcal{O}(d)$.
\item Stability: The linearization of \revised{the lattice dynamical system} (\ref{LDST}) about these solutions posed on $\ell^\infty(\mathbb{Z}^2)$ has precisely eight eigenvalues when $M\neq N-1$ and four eigenvalues when $M=N-1$ that cross the origin as the branch is traversed near each fold bifurcation, and these eigenvalues cross transversely.
\end{compactitem}
\end{lem}

\begin{proof}
We will fix $(N,M)$ and construct $D_4$-symmetric solutions of \revised{the steady-state system} (\ref{LDS}) near the pattern
\[
\bar{u}^{(N,M)}(0)=\bar{v}^{(N,M)}(0) = \left\{\begin{array}{ll}
1 & 1\leq m\leq n<N \\
1 & n=N,\; 1\leq m\leq M, \\
0 & \mathrm{otherwise}
\end{array}\right.
\]
for $(\mu,d)$ near zero. We use $D_4$-symmetry to reduce patterns to the index set $I$ defined in (\ref{e:I}): our proof will apply to both on-site and off-site solutions, and we therefore do not distinguish these cases in what follows. As illustrated in Figures~\ref{f:5} and~\ref{f:11}, we expect that, as the branch passes near $\mu=0$, the cell $u_{N,M+1}$ changes from 0 to $u_-(\mu)$, while the remaining cells will stay near $0$ or $u_+(\mu)$. To capture this behavior, we define the pairwise disjoint index sets
\begin{equation}\label{e:i}
I_+ = \{ (n,m)\in I: \bar{u}^{N,M}_{n,m}=1 \},\quad
I_- = \{ (N,M+1) \}, \quad
I_0 = \{ (n,m)\in I\setminus I_-: \bar{u}^{N,M}_{n,m}=0 \}.
\end{equation}
To solve $\mathcal{F}(u,\mu,d)=0$, we note that $\mathcal{F}(\bar{u}^{(N,M)}(0),0,0)=0$ and that the linearization of $\mathcal{F}$ is given by
\[
(\mathcal{F}_u(\bar{u}^{(N,M)}(0),0,0)v)_{n,m} = \left\{ \begin{array}{cl}
f_u(1,0) v_{n,m} & (n,m)\in I_+ \\
0                & (n,m)\in I\setminus I_+.
\end{array}\right.
\]
Writing $u^+:=u|_{I_+}$ and $u^c:=u|_{I\setminus I_+}$, and using that $f_u(1,0)\neq0$, we can apply the implicit function theorem to conclude that $\mathcal{F}(u,\mu,d)=0$ restricted to the index set $I_+$ has a unique solution $u^+(u^c,\mu,d)\in\ell^\infty(I_+)$ for each $u^c\in\ell^\infty(I\setminus I_+)$ and $(\mu,d)$ near zero, and this solution depends smoothly on its arguments. In particular, we have 
\begin{equation}\label{PitchSol}
u^+(u^c,\mu,d) = 1 + \mathcal{O}(|\mu|+|d| \|u^c\|_\infty).
\end{equation}
To solve \revised{the steady-state system} (\ref{LDS}) on the index set $I\setminus I_+$, we introduce the scaling
\begin{equation}\label{e:scaling}
\mu = \nu^2, \qquad
d = \nu^3 \tilde{d}, \qquad
u_{n,m} = \nu^{\rho_{n,m}} \tilde{u}_{n,m}, \qquad
\rho_{n,m}:=\inf_{(\tilde{n},\tilde{m})\in I_+}(|n-\tilde{n}|+|m-\tilde{m}|)
\end{equation}
for $(n,m)\in I\setminus I_+$ with $|\nu|\ll1$, where $\rho_{n,m}$ is the $\ell^1$ distance of $(n,m)$ from $I_+$. Substituting these expressions into \revised{the steady-state system} (\ref{LDS}), we see that (\ref{LDS}) at index $(n,m)=(N,M+1)$ becomes
\[
0 = \nu^3\tilde{d}u^+_{N-1,M+1} + \nu^3\tilde{d}u^+_{N,M} + \nu^5\tilde{d}u^+_{N+1,M+1} - 4\nu^3\tilde{d}\tilde{u}_{N,M+1} + f(\nu\tilde{u}_{N,M+1},\nu^2) + \left\{
\begin{array}{ll}
\nu^4\tilde{d}\tilde{u}_{N,M+2} & M < N-1 \\
\nu^5\tilde{d}\tilde{u}_{N,M+2} & M = N-1,
\end{array}\right.
\]
where $u^+=u^+((\nu^{\rho_{n,m}}\tilde{u}_{n,m})_{(n,m)\in I\setminus I_+},\nu^2,\nu^3\tilde{d})$. Using (\ref{NormalForm1}) and (\ref{PitchSol}), we find that \revised{the steady-state system} (\ref{LDS}) at index $(n,m)=(N,M+1)$ is given by
\[
0 = \nu^3(2\tilde{d} - \tilde{u}_{N,M+1} + \tilde{u}^3_{N,M+1}) + \mathcal{O}(\nu^4) \qquad\mbox{or}\qquad
0 = 2\tilde{d} - \tilde{u}_{N,M+1} + \tilde{u}^3_{N,M+1} + \mathcal{O}(\nu).
\]
Proceeding in the same fashion for each $(n,m)\in I\setminus I_+$, we see that \revised{the steady-state system} (\ref{LDS}) restricted to the index set $I\setminus I_+$ becomes
\[
\begin{array}{lcl}
n = N, m = M+1:					&& 0 = \nu^3(2\tilde{d} - \tilde{u}_{N,M+1} + \tilde{u}^3_{N,M+1}) + \mathcal{O}(\nu^4) \\
n = N, m = N:						&& 0 = \nu^3(-\tilde{u}_{n,m} + \tilde{u}^3_{n,m}) + \mathcal{O}(\nu^4) \\
n = N, M+1 < m < N:			&& 0 = \nu^3(\tilde{d} - \tilde{u}_{n,m} + \tilde{u}^3_{n,m}) + \mathcal{O}(\nu^4) \\
n = N+1, 1 \leq m < M:	&& 0 = \nu^3(u_+(0)\tilde{d} - \tilde{u}_{n,m} + \tilde{u}^3_{n,m}) + \mathcal{O}(\nu^4) \\
n = N+1,m = M: 					&& 0 = \nu^3(\tilde{d} - \tilde{u}_{n,m} + \tilde{u}^3_{n,m}) + \mathcal{O}(\nu^4) \\
n = N+1,m \geq M+1: 		&& 0 = \nu^4(\tilde{d}\tilde{u}_{n-1,m} - \tilde{u}_{n,m}) + \mathcal{O}(\nu^5) \\
n > N+1,1 \leq m \leq N:&& 0 = \nu^{\rho_{n,m}}(\tilde{d}\tilde{u}_{n-1,m} - \tilde{u}_{n,m}) + \mathcal{O}(\nu^{{\rho_{n,m} + 1}}) \\
n > N+1,m \geq N+1: 		&& 0 = \nu^{\rho_{n,m}}(\tilde{d}\tilde{u}_{n-1,m} + d\tilde{u}_{n,m-1} - \tilde{u}_{n,m}) + \mathcal{O}(\nu^{{\rho_{n,m} + 1}}). \\
\end{array}
\]
Upon dividing by the leading factors in $\nu$, we arrive at the system
\begin{equation}\label{PitchSol2.5}
\begin{array}{lclcl}
(1) && n = N, m = M+1:					&& 0 = 2\tilde{d} - \tilde{u}_{N,M+1} + \tilde{u}^3_{N,M+1} + \mathcal{O}(\nu) \\
(2) && n = N, m = N:						&& 0 = -\tilde{u}_{n,m} + \tilde{u}^3_{n,m} + \mathcal{O}(\nu) \\
(3) && n = N, M+1 < m < N:			&& 0 = \tilde{d} - \tilde{u}_{n,m} + \tilde{u}^3_{n,m} + \mathcal{O}(\nu) \\
(4) && n = N+1, 1 \leq m < M:		&& 0 = \tilde{d} - \tilde{u}_{n,m} + \tilde{u}^3_{n,m} + \mathcal{O}(\nu) \\
(5) && n = N+1,m = M:						&& 0 = \tilde{d} - \tilde{u}_{n,m} + \tilde{u}^3_{n,m} + \mathcal{O}(\nu) \\
(6) && n = N+1,m \geq M+1:			&& 0 = \tilde{d}\tilde{u}_{n-1,m} - \tilde{u}_{n,m} + \mathcal{O}(\nu) \\
(7) && n > N+1,1 \leq m \leq N:	&& 0 = \tilde{d}\tilde{u}_{n-1,m} - \tilde{u}_{n,m} + \mathcal{O}(\nu) \\
(8) && n > N+1,m \geq N+1:			&& 0 = \tilde{d}\tilde{u}_{n-1,m} + d\tilde{u}_{n,m-1} - \tilde{u}_{n,m} + \mathcal{O}(\nu). \\
\end{array}
\end{equation}
Recall that we are interested in a solution branch that connects $u_{N,M+1}=0$ to $u_{N,M+1}=u_-(\mu)$ for $0<d\ll1$, while all other cells $u_{n,m}$ with $(n,m)\in I_0$ remain close to zero. In the scaling we introduced above, it suffices to find a branch that connects $(\tilde{u}_{N,M+1},\tilde{d})=(0,0)$ to $(\tilde{u}_{N,M+1},\tilde{d})=(1,0)$, while the remaining cells $\tilde{u}_{n,m}$ with $(n,m)\in I_0$ are zero when $\tilde{d}$ vanishes. To find this branch, we first set $\nu=0$. In this case, (\ref{PitchSol2.5})(1) has the solution
\begin{equation}\label{e:s1}
(\tilde{u}_{N,M+1},\tilde{d})(s) := \left(s,\frac{s}{2}(1-s^2)\right), \qquad 0\leq s\leq1,
\end{equation}
which connects, as desired, $(0,0)$ to $(1,0)$ as $s$ varies from 0 to 1. Note that the branch (\ref{e:s1}) exhibits a generic fold bifurcation at $s=\frac{1}{\sqrt{3}}$ where $\tilde{d}_\mathrm{sn}:=\frac{1}{3\sqrt{3}}$. When $\tilde{d}=0$, the remaining equations in (\ref{PitchSol2.5}) admit the solution $\tilde{u}_{n,m}=0$ for indices in $I_0$, and it remains to continue this solution for $\tilde{d}\in[0,\tilde{d}_\mathrm{sn}]$. Clearly, $\tilde{u}_{n,m}=0$ is a solution of (\ref{PitchSol2.5})(2) for all $\tilde{d}$. Next, we see that there is a regular smooth solution branch $\tilde{u}_{n,m}(s)$ of (\ref{PitchSol2.5})(3-5) with $\tilde{d}=\tilde{d}(s)$ given by (\ref{e:s1}) that starts at $\tilde{u}_{n,m}=0$ at $s=0$: the end point of the solution branch at $s=1$ is again $\tilde{u}_{n,m}=0$: indeed, this branch never encounters the fold point present in (\ref{PitchSol2.5})(1), since the graph of the function $g(u,a)=a-u+u^3$ increases strictly in $a$, and $\tilde{d}$ has an additional factor two in (\ref{PitchSol2.5})(1). The remaining equations (\ref{PitchSol2.5})(6-8) are of the form
\begin{equation}\label{e:s2}
(-1+\tilde{d}B)\hat{u} = \tilde{d} h
\end{equation}
where $\hat{u}$ consists of the elements $\tilde{u}_{n,m}$ with indices $(n,m)$ listed in the middle columns of (\ref{PitchSol2.5})(6-8), $B$ is linear with $\|B\|\leq2$, and $h$ is a given vector that satisfies $|h|_\infty\leq2$, may depend on $s$, and comprises contributions from the remaining elements $\tilde{u}_{n,m}$. It is now easy to see that the operator on the left-hand side of (\ref{e:s2}) is invertible for all $\tilde{d}\in[0,\tilde{d}_\mathrm{sn}]$ since $2\tilde{d}_\mathrm{sn}<1$, and we conclude that (\ref{PitchSol2.5})(6-8) has a unique solution for all $s\in[0,1]$ that vanishes at $s=0$.

To prove persistence and uniqueness of this branch and the generic fold bifurcation for $0<\nu\ll1$, we write (\ref{PitchSol2.5}) as $G(\tilde{u}^c,\tilde{d},\nu)$ with $\tilde{u}^c\in\ell^\infty(I\setminus I_+)$ and note that the branch $(\tilde{u}^c,\tilde{d})(s)$ constructed above satisfies $G(\tilde{u}^c(s),\tilde{d}(s),0)=0$ and that the derivative $G_{(\tilde{u}^c,\tilde{d})}(\tilde{u}^c(s),\tilde{d}(s),0)$ has full rank for all $s\in[0,1]$. We can therefore apply the implicit function theorem and use persistence results for fold bifurcations to conclude that the branch persists for each $0<\nu<\nu_0$ and that it exhibits a unique generic fold bifurcation at $\tilde{d}=\tilde{d}_\mathrm{sn}(\nu)$ for some smooth function $\tilde{d}_\mathrm{sn}(\nu)$ with $\tilde{d}_\mathrm{sn}(0)=\frac{1}{3\sqrt{3}}$.

The preceding arguments establish the first two statements of the lemma, and it remains to prove the assertion about the eigenvalues of the linearization. We denote the solutions we constructed above by $u^*(s,\nu)=(u_{n,m}^*(s,\nu))_{n,m}$ and extend these solutions and the index sets $I_{\pm,0}$ we defined in (\ref{e:i}) using the underlying $D_4$-symmetry to indices $(n,m)$ in $\mathbb{Z}^2$ rather than just in the wedge $I$. Using the scaling introduced in (\ref{e:scaling}), the linearization of \revised{the lattice dynamical system} (\ref{LDST}) about $u^*(s,\nu)$ is given by
\[
(L(s,\nu) v)_{n,m} := \nu^3 \tilde{d}(s,\nu) (\Delta v)_{n,m} + f^\prime(u^*_{n,m}(s,\nu),\nu^2) v_{n,m}, \quad v\in\ell^\infty(\mathbb{Z}^2).
\]
Setting $\nu=0$, we see that
\[
L(s,0) = f^\prime(u^*_{n,m}(s,0),0) = f^\prime(\bar{u}^{N,M}_{n,m},0) =
\left\{\begin{array}{lcl}
f^\prime(1,0) < 0 && (n,m)\in I_+ \\
0                 && (n,m) \in I_0 \cup I_-.
\end{array}\right.
\]
We see that $L(s,0)$ has spectrum at $f^\prime(1,0)<0$ and at $0$ and that the spectral eigenspace associated with the spectrum at the origin is given by
\[
E^c := \left\{ v\in\ell^\infty(\mathbb{Z}^2):\; u_{n,m}=0 \mbox{ for } (n,m)\in I_+ \right\}.
\]
Since the bounded operator $L(s,\nu)$ depends smoothly on $(s,\nu)$, we can use the $(s,\nu)$-dependent spectral projections to represent the operator $L(s,\nu)$ on the $(s,\nu)$-dependent spectral space corresponding to the spectrum near the origin by a reduced operator posed on $E^c$. Denoting this operator by $L^c(s,\nu)$, we find that it is of the form
\[
L^c(s,\nu) = \nu^2\tilde{L}^c(s,\nu), \qquad
(\tilde{L}^c(s,\nu) v)_{n,m} = \left( -1+3(\tilde{u}^*_{n,m}(s,\nu))^2 \right) v_{n,m} + \mathcal{O}(\nu) v, \qquad v\in E^c.
\]
The bounded operator $\tilde{L}^c(s,\nu)$ on $E^c$ depends smoothly on $(s,\nu)$, and we have
\[
(\tilde{L}^c(s,0) v)_{n,m} = \left( -1+3(\tilde{u}^*_{n,m}(s,0))^2 \right) v_{n,m} =
\left\{\begin{array}{lcl}
(-1+3s^2) v_{n,m} && (n,m)\in I_- \\
\left( -1+3(\tilde{u}^*_{n,m}(s,0))^2 \right) v_{n,m} =: g_{n,m}(s) v_{n,m} && (n,m)\in I_0
\end{array}\right.
\]
for $0\leq s\leq1$. Our existence proof implies that there is a constant $C>0$ so that $g_{n,m}(s)\leq-2C<0$ for all $(n,m)\in I_0$ and $s\in[0,1]$. In particular, $\tilde{L}^c(s,\nu)$ has precisely $|I_-|$ eigenvalues that cross the imaginary axis transversely at the origin as $s$ passes through $1/\sqrt{3}$, while the remaining spectrum lies on or to the left of the line $\Re\lambda=-2C$. Thus, there is a $\nu_*>0$ so that $L^c(s,\nu)$ has, for each $\nu\in(0,\nu_*)$, precisely $|I_-|$ eigenvalues that cross the imaginary axis transversely at the origin near $s=1/\sqrt{3}$ (we remark that the representation of $D_4$ on the $|I_-|$-dimensional eigenspace enforces that eigenvalues cross at the origin and not just near it), while the remaining spectrum of $L^c(s,\nu)$ on $E^c$ lies to the left of the line $\Re\lambda=-C\nu^2$ for $s\in[0,1]$ and $\nu\in(0,\nu_*)$. Finally, Figure~\ref{f:6} shows that $|I_-|$ is equal to eight when $M\neq N-1$ and equal to four when $M=N-1$. This completes the proof of the lemma.
\end{proof}

When $M=N$, the cell at index $(N+1,1)$ will change from $0$ to $u_-(\mu)$ as we pass near $\mu=0$. In this case, we can proceed as above up until (\ref{PitchSol2.5}): this equation remains the same except that the equation at index $(n,m)=(N+1,1)$ becomes
\[
0 = \tilde{d} - \tilde{u}_{N+1,1} + \tilde{u}^3_{N+1,1} + \mathcal{O}(\nu),
\]
since $u_{N+1,1}$ only has one neighbor belonging to $I_1$. Thus, the equation for the critical cell at $(n,m)=(N+1,1)$ is exactly the same as those for the noncritical cells, and without going to higher-order expansions it is not clear what the solution structure is. We will now show that the results of Lemma~\ref{lem:Pitchfork} can be extended to the case $N=M$ for $N=1,2$.

\begin{lem}\label{lem:Pitchfork2}
Fix $N = 1,2$, then the following is true separately for on-site and off-site $D_4$-symmetric solutions of \revised{the steady-state system} (\ref{LDS}). There are constants $d_1,\mu_1 > 0$ and a smooth function $\mu_l:[0,d_l] \to [0,\mu_1]$ such that the following is true for each fixed $d \in (0,d_1]$:
\begin{compactitem}
\item Fold bifurcations: There is a pair of symmetric solutions $u_l(\mu,d)$ and $v_l(\mu,d)$ of \revised{the steady-state system} (\ref{LDS}) that bifurcate at a fold bifurcation at $\mu = \mu_l(d)$ and exist for all $\mu \in [\mu_l(d),\mu_1]$. These solutions are smooth in $(\mu,d)$, and for each fixed $\mu$ we have $u_l(\mu,d) \to \bar{u}^{(N,N)}(\mu)$ and $v_l(\mu,d) \to \bar{v}^{(N+1,1)}(\mu)$ as $d \searrow0$.
\item Expansion: The function $\mu_l(d)$ satisfies $\mu_l(d) = \frac{3}{\sqrt[3]{4}} d^\frac{2}{3} + \mathcal{O}(d)$.
\item Stability: The linearization of \revised{the lattice dynamical system} (\ref{LDST}) about these solutions posed on $\ell^\infty(\mathbb{Z}^2)$ has precisely eight eigenvalues that cross the origin as the branch is traversed near each fold bifurcation, and these eigenvalues cross transversely.
\end{compactitem}
\end{lem}

\begin{proof}
First, we consider off-site solutions with $N=1$. Proceeding as in Lemma~\ref{lem:Pitchfork} we arrive at
\begin{equation}\label{PitchSol3}
\begin{array}{lclcl}
(1) && n = 2, m = 1:	&& 0 = \tilde{d} - \tilde{u}_{2,1} + \tilde{u}^3_{2,1} + \mathcal{O}(\nu) \\
(2) && n = 2, m = 2: && 0 = 2\tilde{d}\tilde{u}_{1,2} - \tilde{u}_{2,2} + \mathcal{O}(\nu) \\
(3) && n > 2,1 \leq m \leq 2: && 0 = \tilde{d}\tilde{u}_{n-1,m} - \tilde{u}_{n,m} + \mathcal{O}(\nu) \\
(4) && n > 2,m \geq 3: & &0 = \tilde{d}\tilde{u}_{n-1,m} + d\tilde{u}_{n,m-1} - \tilde{u}_{n,m} + \mathcal{O}(\nu),
\end{array}
\end{equation}
where we seek a solution branch that connects $\tilde{u}_{2,1}=0$ to $\tilde{u}_{2,1}=1$, while the other cells $\tilde{u}_{n,m}$ with $n>2$ or $(n,m)=(2,2)$ remain close to zero. The key is that there is only a single leading-order equation, namely (\ref{PitchSol3})(i) for the critical index $(2,1)$, and we can therefore proceed as in the proof of Lemma~\ref{lem:Pitchfork}; we omit the details.

Next, we consider off-site solutions with $N=2$ and seek a solution branch that connects $u_{3,1}=0$ to $u_{3,1}=u_+(\mu)$. We proceed as before but now encounter the situation that the equations for two cells, including the critical cell $u_{3,1}$, coincide to leading order. We therefore include the next-order terms in $\nu$ in these equation and, as we will explain below, arrive at the system
\begin{equation}\label{PitchSol4}
\begin{array}{lclcl}
(1) && n = 3, m = 1: && 0 = \tilde{d} - \tilde{u}_{3,1} + \tilde{u}^3_{3,1} + \nu\tilde{d} (\tilde{u}_{3,2}-3\tilde{u}_{3,1}) + \mathcal{O}(\nu^\frac{3}{2}) \\
(2) && n = 3, m = 2: && 0 = \tilde{d} - \tilde{u}_{3,2} + \tilde{u}^3_{3,2} + \nu\tilde{d}(\tilde{u}_{3,1}-4\tilde{u}_{3,2}) + \mathcal{O}(\nu^\frac{3}{2}) \\
(3) && n = 3, m = 3: && 0 = \tilde{d}\tilde{u}_{3,2} - \tilde{u}_{3,3} + \mathcal{O}(\nu) \\
(4) && n > 3,1 \leq m \leq 2:	&& 0 = \tilde{d}\tilde{u}_{n-1,m} - \tilde{u}_{n,m} + \mathcal{O}(\nu) \\
(5) && n > 3,m \geq 3: && 0 = \tilde{d}\tilde{u}_{n-1,m} + d\tilde{u}_{n,m-1} - \tilde{u}_{n,m} + \mathcal{O}(\nu).
\end{array}
\end{equation}
Note that equations (1) and (2) in (\ref{PitchSol4}) agree to lowest order in $\nu$ but differ at the next order. The differences arise from the discrete Laplace operator: the equation for $\tilde{u}_{3,1}$ contains the term $-3\tilde{u}_{3,1}$ at order $\nu$ since the off-site symmetry enforces $u_{3,0} = u_{3,1}$, thus eliminating one of the connections at the index $(n,m)=(3,1)$. In contrast, we have exactly four self-interactions in (\ref{PitchSol4})(2) at order $\nu$ since the element at index $(n,m)=(3,2)$ has no neighbors that have a symmetric restriction imposed them. We set
\[
\tilde{d} = \frac{2}{3\sqrt{3}}+ \nu d_0, \quad
\tilde{u}_{3,1} = \frac{1}{\sqrt{3}} + \nu^\frac{1}{2}v_1, \quad
\tilde{u}_{3,2} = \frac{1}{\sqrt{3}} +\nu^\frac{1}{2}v_2
\]
so that (\ref{PitchSol4})(1-2) become
\[
\begin{array}{lclcl}
(1) && n = 3, m = 1: && 0 = d_0 + \sqrt{3}v_1^2 - \frac{4}{27} + \mathcal{O}(\nu^\frac{1}{2}) \\
(2) && n = 3, m = 2: && 0 = d_0 + \sqrt{3}v_2^2 - \frac{2}{9} + \mathcal{O}(\nu^\frac{1}{2})
\end{array}
\]
after dividing out the leading-order factor in $\nu$, while (3)-(5) remain unchanged to leading order. Setting $\nu = 0$, we see that (\ref{PitchSol4})(1) has the solution
\[
(v_1,d_0)(s) = \left(s,\frac{4}{27} - \sqrt{3}s^2\right)
\]
that connects $(v_1,d_0)=(-\frac{2\sqrt[4]{3}}{9},0)$ to $(v_1,d_0)=(\frac{2\sqrt[4]{3}}{9},0)$ and exhibits a generic fold bifurcation at $s=0$ where $d_0=\frac{4}{27}$. Note that the solution of (\ref{PitchSol4})(2) stays on the left solution branch as the fold bifurcation for this equation occurs for a larger value of $d_0$. We can now proceed as in the proof of Lemma~\ref{lem:Pitchfork} to complete the existence proof for off-site solutions. The stability proof is also very similar to proof of Lemma~\ref{lem:Pitchfork} except that we need to carry out two reduction steps using spectral projections: As before, the first step separates the $\mathcal{O}(\nu^2)$ spectrum from the spectrum at $f^\prime(1,0)$. The second step separates spectrum to the left of the line $\Re\lambda\leq-C\nu^2$ from the eigenvalues $\lambda\approx\nu^{\frac{5}{2}}$ that arise from (\ref{PitchSol4})(1-2). Scaling eigenvalues with $\nu^{\frac{5}{2}}$ on the $2|I_-|$-dimensional eigenspace obtained in the second step then shows that only $|I_-|$ of them cross transversely, while the others stay to the left of the imaginary axis.

It remains to consider on-site solutions. The proof for the $N=1$ on-site solution is identical to the proof above, while the only difference for the $N=2$ case is that the reduced equations that agree to leading order are now given by
\[
\begin{array}{lclcl}
(1) && n = 3, m = 1: && 0 = \tilde{d} - \tilde{u}_{3,1} + \tilde{u}^3_{3,1} + \nu(2\tilde{d}u_{3,2}-4\tilde{d}u_{3,1}) + \mathcal{O}(\nu^\frac{3}{2}) \\
(2) && n = 3, m = 2: && 0 = \tilde{d} - \tilde{u}_{3,2} + \tilde{u}^3_{3,2} + \nu(\tilde{d}u_{3,1}-4\tilde{d}u_{3,2}) + \mathcal{O}(\nu^\frac{3}{2}),
\end{array}
\]
where we used that on-site symmetry implies that $u_{3,0}=u_{3,2}$. The remainder of the proof is similar to the off-site case studied above.
\end{proof}

\begin{rmk}\label{r:1}
Our proof of Lemma~\ref{lem:Pitchfork2} indicates that the equations for the critical cell with index $(n,m)=(N+1,1)$ and the neighboring cell with index $(n,m)=(N+1,2)$ for $N\geq3$ agree up to terms of any order less than $\sum_{j=1}^N \frac{1}{2^j}$ in $\nu$. We were not able to find a consistent pattern that allowed us to write down the resulting equations and solve them simultaneously in $N\geq3$. 
\end{rmk}


\subsection{Continuation through bifurcations near $\mu=1$}\label{s3.4}

We continue branches through $\mu=1$: Lemma~\ref{lem:Saddle2} provides results for $\bar{u}^{(N,1)}$ with $N\geq3$, while Lemma~\ref{lem:Saddle} deals with the cases of $\bar{u}^{(N,N-1)}$ and $\bar{u}^{(N,N)}$ with $N\geq2$. Note that we cannot establish continuation through $\mu=1$ for $\bar{u}^{(N,M)}$ when $1\leq M\leq N-2$). Changing coordinates, we can bring the Taylor expansion of $f(u,\mu)$ about $(1,1)$ into the form
\[
f(1+u,1+\mu) = -\mu - u^2 + b_1 \mu u + b_2 u^3 + \mathcal{O}(\mu^2 + \mu u^2 + u^4).
\]
We now state our first result.

\begin{lem}\label{lem:Saddle}
Fix $N \geq 2$ and $M\in\{N-1,N\}$, then the following is true separately for on-site and off-site $D_4$-symmetric solutions of \revised{the steady-state system} (\ref{LDS}). There exist constants $d_2,\mu_2 > 0$ and a smooth function $\mu_r:[0,d_2] \to [\mu_2,1]$ such that the following is true for each fixed $d \in (0,d_2]$:
\begin{compactitem}
\item Fold bifurcation: There is a pair of symmetric solutions $u_r(\mu,d)$ and $v_r(\mu,d)$ of \revised{the steady-state system} (\ref{LDS}) that bifurcate at a fold bifurcation at $\mu = \mu_r(d)$ and exist for all $\mu \in [\mu_2,\mu_r(d)]$. These solutions are smooth in $(\mu,d)$, and for each fixed $\mu$ we have $u_r(\mu,d) \to \bar{u}^{(N,M)}(\mu)$ and $v_r(\mu,d) \to \bar{v}^{(N,M)}(\mu)$ as $d \searrow0$.
\item Expansion: The function $\mu_r(d)$ is given by $\mu_r(d)=1-2d+\mathcal{O}(d^\frac{3}{2})$.
\item Stability: The linearization of \revised{the lattice dynamical system} (\ref{LDST}) about these solutions posed on $\ell^\infty(\mathbb{Z}^2)$ has precisely eight eigenvalues when $M=N-1$ and four eigenvalues when $M=N$ that cross the origin as the branch is traversed, and these eigenvalues cross transversely.
\end{compactitem}
\end{lem}

\begin{proof}
Our proof will apply to both on-site and off-site solutions, and we therefore do not distinguish these case in what follows. We again use $D_4$-symmetry to reduce patterns to the index set $I$ defined in (\ref{e:I}). Figures~\ref{f:5} and \ref{f:11} indicate that, as the branch passes near $\mu = 1$, the cell $u_{N,M}$ changes from $u_-(\mu)$ to $u_+(\mu)$, while the remaining cells stay near $0$ or $u_+(\mu)$. We therefore define the index sets
\[
I_- =\{(N,M)\}, \quad
I_+ =\{(n,m)\in I\setminus I_-: \bar{u}_{n,m}^{N,M} = 1\}, \quad
I_0 =\{(n,m)\in I: \bar{u}_{n,m}^{N,M} = 0\}.
\]
We have that $\mathcal{F}(\bar{u}^{(N,M)}(1),1,0)=0$ and that the linearization of $\mathcal{F}$ about this solution is given by 
\[
(\mathcal{F}_u(\bar{u}^{(N,M)}(1),1,0)v)_{n,m} =
\left\{ \begin{array}{cl}
f_u(0,1) v_{n,m} & (n,m)\in I_0 \\
0                & (n,m)\in I\setminus I_0.
\end{array}\right.
\]
Writing $u^0:=u|_{I_0}$ and $u^c:=u|_{I\setminus I_0}$, and using that $f_u(0,1)\neq0$, we can apply the implicit function theorem to find that $\mathcal{F}(u,\mu,d) = 0$ restricted to the index set $I_0$ has a unique solution $u^0(u^c,\mu,d) \in \ell^\infty(I_0)$ for each $u^c \in \ell^\infty(I\setminus I_0)$ and $(\mu,d)$ near $(1,0)$. This solution depends smoothly on its arguments, and in particular, has the expansion
\begin{equation}\label{Saddle1}
u^0(u^c,\mu,d) = \mathcal{O}(|\mu - 1| + |d|\|u^c\|_\infty).
\end{equation}
To solve \revised{the steady-state system} (\ref{LDS}) on the index set $I\setminus I_0$, we introduce the scaling
\[
\mu = 1 -\nu^2, \quad d = \nu^2\tilde{d}, \quad u_{n,m} = 1 + \nu\tilde{u}_{n,m},
\]
where $(n,m)\in I\setminus I_0$ and $|\nu| \ll 1$. Expanding $\mathcal{F}(u,\mu,d)=0$ restricted to the index set $I\setminus I_0$ in powers of $\nu$ and dividing by the leading factor in $\nu$, we arrive at the finite system
\begin{equation}\label{Saddle2}
\begin{array}{lclcl}
(1) && n = N, m = M: && 0 = -2\tilde{d} + 1 - \tilde{u}_{N,M}^2 + \mathcal{O}(\nu) \\
(2) && n = N,1 \leq m < M: && 0 = -\tilde{d} + 1 - \tilde{u}_{N,m}^2 + \mathcal{O}(\nu) \\
(3) && n < N:	&& 0 = 1 - \tilde{u}_{n,m}^2 + \mathcal{O}(\nu),
\end{array}
\end{equation}
where we used (\ref{Saddle1}) to simplify the equations with $n=N$. We seek a solution branch that connects $u_{N,M} = u_-(\mu)$ to $u_{N,M} = u_+(\mu)$ for $0 < d \ll 1$, while all other cells $u_{n,m}$ with $(n,m) \in I_+$ remain close to $u_+(\mu)$. Using our scaling, this means finding solutions of (\ref{Saddle2}) that connect $(\tilde{u}_{N,M},\tilde{d})=(-1,0)$ to $(\tilde{u}_{N,M},\tilde{d})=(1,0)$, while the remaining cells $\tilde{u}_{n,m}$ are $1$ for $\tilde{d}=0$. For $\nu=0$, (\ref{Saddle2})(1) has the solution
\[
(\tilde{u}_{N,M},\tilde{d})(s) = \left(s,\frac{1}{2}(1-s^2)\right)
\quad\mbox{for}\quad -1 \leq s \leq 1,
\]  
which connects $(-1,0)$ at $s=-1$ to $(1,0)$ at $s=1$ and exhibits a generic fold bifurcation at $s=0$. We can now follow the proof of Lemma~\ref{lem:Pitchfork} to complete the proof.
\end{proof}

Note that this lemma does not cover the case of indices $(N,M)$ with $1\leq M<N-1$ for $N\geq3$. For off-site solutions (the case of on-site patterns is similar), the reason is again that the rescaled equations the indices $(N,M)$ and $(N-1,N-1)$ are identical. Indeed, following the same process as above, we obtain the system
\begin{equation}\label{Saddle3}
\begin{array}{lclcl}
(1) && n = N, m = M: && 0 = -2\tilde{d} +1 - \tilde{u}_{N,M}^2 + \mathcal{O}(\nu) \\
(2) && n = m = N-1:	&& 0 = -2\tilde{d} +1 -\tilde{u}_{N-1,N-1}^2 + \mathcal{O}(\nu) \\
(3) && n = N,1 \leq m < M: && 0 = -\tilde{d} + 1 -\tilde{u}_{N,m}^2 + \mathcal{O}(\nu) \\
(4) && n = N-1, M< m < N-1: && 0 = -\tilde{d} + 1 -\tilde{u}_{N-1,m}^2 + \mathcal{O}(\nu) \\
(5) && n = N-1, 1 \leq m \leq M	&& 0 = 1 -\tilde{u}_{N-1,m}^2 + \mathcal{O}(\nu) \\
(6) && n \leq N-2 && 0 = 1 -\tilde{u}_{n,m}^2 + \mathcal{O}(\nu),
\end{array}
\end{equation}
and we would again need to identify the higher-order corrections. The following lemma addresses this issue for $M=1$.

\begin{lem}\label{lem:Saddle2}
Fix $N\geq3$ and $M=1$, then the following is true separately for on-site and off-site $D_4$-symmetric solutions of \revised{the steady-state system} (\ref{LDS}). There are constants $d_2,\mu_2 > 0$ and a smooth function $\mu_r:[0,d_2] \to [\mu_2,1]$ such that the following is true for each fixed $d \in (0,d_2]$:
\begin{compactitem}
\item Fold bifurcations: There is a pair of symmetric solutions $u_r(\mu,d)$ and $v_r(\mu,d)$ of \revised{the steady-state system} (\ref{LDS}) that bifurcate at a saddle-node bifurcation at $\mu = \mu_r(d)$ and exist for all $\mu \in [\mu_2,\mu_r(d)]$. These solutions are smooth in $(\mu,d)$, and for each fixed $\mu$ we have $u_r(\mu,d) \to \bar{u}^{(N,1)}(\mu)$ and $v_r(\mu,d) \to \bar{v}^{(N,1)}(\mu)$ as $d \searrow0$.
\item Expansion: The function $\mu_r(d)$ is given by $\mu_r(d)=1-2d+\mathcal{O}(d^\frac{3}{2})$. 
\item Stability: The linearization of \revised{the lattice dynamical system} (\ref{LDST}) about these solutions posed on $\ell^\infty(\mathbb{Z}^2)$ has precisely eight eigenvalues that cross the origin as the branch is traversed, and these eigenvalues cross transversely.
\end{compactitem}
\end{lem}

\begin{proof}
For off-site solutions, we proceed as in Lemma~\ref{lem:Saddle} until we arrive at (\ref{Saddle3}). Since $M=1$ implies that (\ref{Saddle3})(3) is not present, we can solve (\ref{Saddle3})(4-6) using the implicit function theorem to arrive at the solution
\begin{equation}\label{SaddleExpansions}
\begin{array}{lcl}
n=N-1, 1<m<N-1: && \tilde{u}_{n,m} = \sqrt{1-\tilde{d}} + \mathcal{O}(\nu) \\
n=N-1, m=1 && \tilde{u}_{N-1,1}= 1 + \mathcal{O}(\nu) \\
n\leq N-2 && \tilde{u}_{n,m} = 1 + \mathcal{O}(\nu),
\end{array}
\end{equation}
which depends smoothly on $(\tilde{u}_{N,1},\tilde{u}_{N-1,N-1})$, $\tilde{d}\in[0,\frac34]$, and $|\nu|\ll1$. The remaining equations (\ref{Saddle3})(1-2) for the indices $(n,m)=(N,1)$ and $(N-1,N-1)$ coincide to leading order, and we therefore expand them to the next highest order in $\nu$. We find that  (\ref{Saddle3})(1) becomes
\begin{eqnarray*}
0 & = & -2\tilde{d} +1 - \tilde{u}_{N,1}^2 + \nu\left(\tilde{d}(\tilde{u}_{N-1,1} + \tilde{u}_{N,0}) - 4\tilde{d}\tilde{u}_{N,1} +b_1\tilde{u}_{N,1} + b_2\tilde{u}_{N,1}^3\right) + \mathcal{O}(\nu^2)
\\ & \stackrel{(\ref{SaddleExpansions})}{=} &
-2\tilde{d} +1 - \tilde{u}_{N,1}^2 + \nu\left(\tilde{d}-3\tilde{d}\tilde{u}_{N,1} +b_1\tilde{u}_{N,1} + b_2\tilde{u}_{N,1}^3\right) + \mathcal{O}(\nu^2),
\end{eqnarray*}
where we used that the $D_4$-symmetry implies $u_{N,0}=u_{N,1}$, while equation (\ref{Saddle3})(2) becomes
\begin{eqnarray*}
0 & = & -2\tilde{d} + 1 -\tilde{u}_{N-1,N-1}^2 + \nu\left(
\tilde{d}(\tilde{u}_{N-2,N-1} + \tilde{u}_{N-1,N-2}) - 4\tilde{d}\tilde{u}_{N-1,N-1}
\right. \\ && \left.
+ b_1\tilde{u}_{N-1,N-1} + b_2\tilde{u}_{N-1,N-1}^3\right) + \mathcal{O}(\nu^2)
\\ & \stackrel{(\ref{SaddleExpansions})}{=} &
-2\tilde{d} + 1 -\tilde{u}_{N-1,N-1}^2 + \nu\left(2\tilde{d}\sqrt{1-\tilde{d}} - 4\tilde{d}\tilde{u}_{N-1,N-1}
\right. \\ && \left.
+ b_1\tilde{u}_{N-1,N-1} + b_2\tilde{u}_{N-1,N-1}^3\right) + \mathcal{O}(\nu^2),
\end{eqnarray*}	
where $u_{N-2,N-1}=u_{N-1,N-2}$ by $D_4$-symmetry. Next, we change variables according to
\[
\tilde{d} = \frac{1}{2} + \nu\tilde{d}_0, \quad
\tilde{u}_{N,1} = \nu^\frac{1}{2}v_1, \quad
\tilde{u}_{N-1,N-1} = \nu^\frac{1}{2}v_2
\]
so that (\ref{Saddle3})(1-2) become 
\[
\begin{array}{lclcl}
(1) && (N,1): && 0 = \frac{1}{2} - 2\tilde{d}_0 - v_1^2 + \mathcal{O}(\nu^\frac{1}{2}) \\
(2) && (N-1,N-1):	&& 0 = \sqrt{2} - 2\tilde{d}_0 - v_2^2 + \mathcal{O}(\nu^\frac{1}{2})
\end{array}	
\] 
after dividing by the leading factor in $\nu$. We can now continue as in the proof of Lemma~\ref{lem:Pitchfork2} to complete the proof of the lemma.

The situation for on-site patterns is actually simpler as the rescaled equation at index $(N,1)$ is given by
\[
0 = -3\tilde{d} +1 - \tilde{u}_{N,M}^2 + \mathcal{O}(\nu)	
\]
while the remaining equations are as in (\ref{Saddle3}). In this case, no further rescaling is necessary, and we can continue as in  Lemma~\ref{lem:Saddle} focusing only on (\ref{Saddle3}).
\end{proof}

\begin{rmk}\label{r:2}
For $1<M<N-1$, we expect that (\ref{Saddle3})(1) and (\ref{Saddle3})(2) differ only at order $\mathcal{O}(\nu^\frac{M+1}{2})$, thus requiring increasingly higher-order expansions in $\nu$ as $M$ increases. This makes it significantly more complicated to generalize Lemma~\ref{lem:Saddle2}.
\end{rmk}


\subsection{Extension to transcritical bifurcations}\label{sec:app}

We consider extensions from pitchfork and fold bifurcations to transcritical bifurcations at $\mu=0,1$. We first comment on the case where the nonlinearity $f(u,\mu)$ in the lattice dynamical system (\ref{LDST}) admits the Taylor expansion
\begin{equation}\label{TaylorSeries2}
f(u,\mu) = -\mu u + u^2 + \mathcal{O}(\mu^2 + \mu u^2 + u^3).
\end{equation}
at $(u,\mu)=(0,0)$, reflecting the existence of a transcritical bifurcation at $\mu=0$ instead of the pitchfork bifurcation we assumed in Hypothesis~\ref{h1}. As in \S\ref{s3.3} we further assume for simplicity that $u_+(0)=1$. The following result extends Lemmas~\ref{lem:Pitchfork} and~\ref{lem:Pitchfork2} to this situation.

\begin{lem}\label{lem:Transcritical_0}
Assume that $f:\mathbb{R}^2\to\mathbb{R}$ has the Taylor expansion (\ref{TaylorSeries2}) at $(u,\mu)=(0,0)$.
\begin{compactenum}[(i)]
\item
For each fixed $N\geq2$ and $1\leq M\leq N-1$ the conclusions of Lemma~\ref{lem:Pitchfork} hold except that the expansion for the location of the fold bifurcations is now given by $\mu_l(d)=2\sqrt{2}d^\frac{1}{2}+\mathcal{O}(d)$.
\item
For $N=1,2$ the conclusions of Lemma~\ref{lem:Pitchfork2} hold except that the expansion for the location of the fold bifurcations is now given by $\mu_l(d)=2d^\frac{1}{2}+\mathcal{O}(d)$.
\end{compactenum}
\end{lem}

\begin{proof}
To prove (i), we proceed as in the proof of Lemma~\ref{lem:Pitchfork} except that we use the variables
\[
u_{n,m} = \mu^{\rho_{n,m}}\tilde{u}_{n,m}, \quad
d = \mu^2\tilde{d}
\]
and arrive at the reduced scaled system
\[
\begin{array}{lcl}
n = N, m = M+1: & \quad & 0 = 2\tilde{d} - \tilde{u}_{N,M+1} + \tilde{u}^2_{N,M+1} + \mathcal{O}(\mu) \\
n = N, M+1 < m < N: && 0 = \tilde{d} - \tilde{u}_{n,m} + \tilde{u}^2_{n,m} + \mathcal{O}(\mu) \\
n = N+1, 1 \leq m < M: && 0 = \tilde{d} - \tilde{u}_{n,m} + \tilde{u}^2_{n,m} + \mathcal{O}(\mu) \\
n = N+1,m = M: && 0 = \tilde{d} - \tilde{u}_{n,m} + \tilde{u}^2_{n,m} + \mathcal{O}(\mu) \\
n = N, m = N: && 0 = -\tilde{u}_{n,m} + \tilde{u}^2_{n,m} + \mathcal{O}(\mu^3) \\
n = N+1,m \geq M+1: && 0 = \tilde{d}\tilde{u}_{n-1,m} - \tilde{u}_{n,m} + \mathcal{O}(\mu) \\
n > N+1,1 \leq m \leq N: && 0 = \tilde{d}\tilde{u}_{n-1,m} - \tilde{u}_{n,m} + \mathcal{O}(\mu) \\
n > N+1,m \geq N+1: && 0 = \tilde{d}\tilde{u}_{n-1,m} + d\tilde{u}_{n,m-1} - \tilde{u}_{n,m} + \mathcal{O}(\mu)
\end{array}
\]
for indices $(n,m)\in I\setminus I_+$. The remainder of the proof is now identical to that of Lemma~\ref{lem:Pitchfork}. We omit the proof of (ii) as it is similar to the proof of Lemma~\ref{lem:Pitchfork2}.
\end{proof}

Next, we comment on the case where the nonlinearity $f(u,\mu)$ admits the Taylor expansion
\begin{equation}\label{TaylorSeries3}
f(1+u,1+\mu) = \mu u - u^2 + \mathcal{O}(\mu^2 + \mu u^2 + u^3)
\end{equation}
at $(u,\mu)=(1,1)$, reflecting the existence of a transcritical bifurcation at $\mu=1$ instead of the saddle-node bifurcation we assumed in Hypothesis~\ref{h1}. The following result extends Lemmas~\ref{lem:Saddle} and~\ref{lem:Saddle2} to this situation. We omit the proof as it is analogous to the proofs given above.

\begin{lem}
Assume that $f:\mathbb{R}\times\mathbb{R}\to\mathbb{R}$ has the Taylor expansion (\ref{TaylorSeries3}) at $(u,\mu)=(1,1)$.
\begin{compactenum}[(i)]
\item
Fix $N\geq2$ and $M\in\{N-1,N\}$, then the conclusions of Lemma~\ref{lem:Saddle} hold except that the fold curve has the expansion $\mu_r(d)=1-2\sqrt{2}d^\frac{1}{2}+ \mathcal{O}(d)$.
\item
For each $N\geq3$, the conclusions of Lemma~\ref{lem:Saddle2} hold except that the fold curve has the expansion $\mu_r(d)=1-\sqrt{2}d^\frac{1}{2}+ \mathcal{O}(d)$.
\end{compactenum}
\end{lem}


\section{Discussion}\label{sec:Discussion}

Motivated in part by the intention to understand the complex bifurcation diagram of hexagon patches in the planar Swift--Hohenberg equation, we studied $D_4$-symmetric patches of ODEs posed on a square lattice. For large coupling strengths, numerical computations (see Figures~\ref{f:1}(ii) and~\ref{f:2}(iii)) show that the bifurcation diagrams as well as the spatial profile changes that occur along these branches are very similar.

Since the case of larger coupling strengths is difficult to tackle analytically, we focused on the anti-continuum limit where the coupling strength is small. In this case, the bifurcation is more regular and indeed resembles the case of spatially one-dimensional patterns (see Figures~\ref{f:1}(i) and~\ref{f:2}(i)). We used Lyapunov--Schmidt reduction to rigorously establish parts of the bifurcation diagram of localized patterns in this regime. We were not able to continue localized patterns through some of the fold bifurcations as this would have necessitated higher-order expansions that we did not carry out in this work.

Taylor and Dawes \cite{Taylor} had explained the transition from the regular diagrams for small coupling strengths to the more complex diagrams that occur for larger coupling strengths by the occurrence of switchbacks, where isolas collide with the primary branch in cusp bifurcations that lead to a rearrangement of the solution branch. We provided a more systematic computation of the cusp bifurcations at the patches $\bar{u}^{(N,1)}$: our main finding is that the cusp bifurcation points we found appear to converge to a specific parameter combination of system parameter and coupling strength as the diameter of the localized patches increases. This indicates that the transition from regular to complex diagrams occurs relatively suddenly as a critical threshold of the coupling strength is crossed. We also showed in Figure~\ref{f:8}(ii) that the switchbacks for patches with larger spatial extent can become quite complex.

\revised{For spatially extended systems posed on cylindrical domains, the location of all branches of asymmetric patterns, including their stability properties, is determined solely by the location of the branches corresponding to symmetric patterns \cite{Avitabile, Beck, Makrides2}. In contrast, it seems impossible to predict the location of asymmetric branches in the planar lattice case based on the location of on-site or off-site $D_4$-symmetric patterns.}

There are several open problems that emerge from this and earlier work. As mentioned above, we were not able to continue the square patches through all folds. We are also not certain whether we captured all possible cusp bifurcations at which the primary branch undergoes transitions: we attempted to find cusps at other patterns besides the $\bar{u}^{(N,1)}$ patches, and while our algorithm failed at these folds, this is not sufficient evidence for the non-existence of cusps elsewhere on the branch. One possible avenue is to continue each fold in $(\mu,d)$-space whilst testing for cusps during continuation.

In this work, we coupled neighboring cells in the lattice using the 5-point Laplacian, and our analysis utilized the resulting coupling structure extensively. It would be interesting to explore other coupling operators with finite and possibly infinite support. Since our analysis in the anti-continuum limit relied on Lyapunov--Schmidt reduction, these cases should be amenable to analysis as well, and it would be interesting to see whether the bifurcation diagrams are similar. Finally, other lattices could be explored: we refer to \revised{\cite{Bramburger, Susanto} for numerical computations of localized patterns on hexagonal lattices, to \cite{Susanto} for triangular lattices,} and to \cite{McCullen} for a numerical study of snaking of localized patterns in a predator-prey model on Barab\'{a}si--Albert networks, where the coupling operator is given by the graph laplacian.

Finally, our work was motivated partially by the behavior of hexagon patches in the Swift--Hohenberg equation, which were studied originally in \cite{Lloyd}. Analysing this case remains out of reach. It would be interesting to compare the intuition from the existence regions of hexagon fronts with different orientations studied in \cite{Lloyd, Kozyreff2} with the case of lattices, where they might be easier to analyse. It would also be interesting to explore the connections with localized quasi-crystals studied more recently in \cite{Subramanian}.


\paragraph{Acknowledgements.}
Bramburger was supported by an NSERC PDF. Sandstede was partially supported by the NSF through grant DMS-1714429.



\end{document}